\documentclass[twoside]{article}

\usepackage[accepted]{aistats2022_arxiv}

\usepackage[
backend=bibtex,
style=alphabetic,
sorting=ynt,
maxnames=100,
maxcitenames=3,
]{biblatex}
\usepackage[algo2e,ruled,vlined]{algorithm2e}
\usepackage{algorithm}%
\usepackage{algpseudocode}%

\usepackage{graphicx}
\graphicspath{ {./figures/} }

\usepackage{caption}
\usepackage{subcaption}

\usepackage[utf8]{inputenc}
\usepackage{xcolor} 

\usepackage{amsmath,amssymb,amsfonts,amsthm}
\usepackage{cancel}
\usepackage{hyperref,cleveref}

\newtheorem{theorem}{Theorem}
\newtheorem{corollary}{Corollary}
\newtheorem{lemma}{Lemma}

\newtheorem{assumption}{Assumption}
\newtheorem{definition}{Definition}
\newtheorem{remark}{Remark}

\def\cL{\mathcal{L}}

\def\cX{\mathcal{X}}
\def\cY{\mathcal{Y}}
\def\cP{\mathcal{P}}
\def\tx{\widetilde{x}}
\def\ty{\widetilde{y}}
\def\tu{\widetilde{u}}
\def\tv{\widetilde{v}}
\def\ux{\underline{x}}
\def\uy{\underline{y}}

\def\tnabla{\widetilde{\nabla}}
\def\gap{\textrm{gap}}
\def\domain{\mathrm{dom}}

\def\uf{{\underline{f}}}
\def\uh{{\underline{h}}}

\def\hy{\widehat{y}}

\def\Ord{\mathcal{O}}

\def\prox{\mathrm{prox}}
\def\Id{\mathbf{I}}

\newcommand{\Ip}[2]{\left\langle#1, #2\right\rangle}

\begin{document}

\runningtitle{Lifted Primal-Dual Method for Bilinearly Coupled Smooth Minimax Optimization}

\twocolumn[

\aistatstitle{
Lifted Primal-Dual Method for \\ Bilinearly Coupled Smooth Minimax Optimization
}

\aistatsauthor{ Kiran Koshy Thekumparampil \And Niao He \And Sewoong Oh }

\aistatsaddress{ 
University of Illinois at Urbana-Champaign \\ \url{thekump2@illinois.edu}
\And 
ETH Z\"urich \\
\url{niao.he@inf.ethz.ch} 
\And 
University of Washington\\
\url{sewoong@cs.washington.edu}
} 
]

\begin{abstract}
We study the bilinearly coupled minimax problem: $\min_{x} \max_{y} f(x) + \Ip{y}{Ax} - h(y)$, where $f$ and $h$ are both strongly convex smooth functions and admit first-order gradient oracles. Surprisingly, no known first-order algorithms have hitherto achieved the lower complexity bound of $\Omega((\sqrt{\frac{L_x}{\mu_x}} + \frac{\|A\|}{\sqrt{\mu_x \mu_y}} + \sqrt{\frac{L_y}{\mu_y}}) \log(\frac1{\varepsilon}))$ for solving this problem up to an $\varepsilon$ primal-dual gap in the general parameter regime, where $L_x, L_y,\mu_x,\mu_y$ are the corresponding  smoothness and strongly convexity constants. 

We close this gap by devising the first \emph{optimal} algorithm, the \emph{Lifted Primal-Dual (LPD) method}. Our method lifts the objective into an extended form that allows both the smooth terms and the bilinear term to be handled optimally and seamlessly with the same primal-dual framework. Besides optimality, our method yields a desirably simple \emph{single-loop} algorithm that uses only one gradient oracle call per iteration. 
Moreover, when $f$ is just convex, the same algorithm applied to a smoothed objective achieves the nearly optimal iteration complexity.
We also
provide a direct single-loop algorithm, using the LPD method, that achieves the iteration complexity of $\Ord(\sqrt{\frac{L_x}{\varepsilon}} + \frac{\|A\|}{\sqrt{\mu_y \varepsilon}} + \sqrt{\frac{L_y}{\varepsilon}})$. 
Numerical experiments on quadratic minimax problems and policy evaluation problems further demonstrate the fast convergence of our algorithm in practice. 
\end{abstract}

\section{Introduction}
\label{sec:intro}

Smooth minimax optimization has gained renewed interest driven by a wide spectrum of applications in machine learning, especially those arising in adversarial training, generative adversarial networks, and reinforcement learning. 
A plethora of first-order algorithms have been developed in the classical  and recent literature, ranging from convex to nonconvex settings, from deterministic to stochastic oracles, from single-loop to multiple-loop schemes.  However, our theoretical understanding of the iteration complexity of minimax optimization  is far from complete even in the canonical  \emph{strongly-convex-strongly-concave (SC-SC)} setting. In particular, the optimal dependence on the condition numbers of different blocks of variables has not been fully characterized. 

Consider the smooth convex-concave minimax problem (a.k.a. saddle point problem): 
\begin{equation}\label{eq:general_minimax}
\min_{x}\max_{y}\; \phi(x,y),
\end{equation}
where $\phi(x,y)$ is $\mu_x$-strongly convex in $x$ and $\mu_y$-strongly concave in $y$. Let $L_x, L_y, L_{xy}$ be the corresponding gradient Lipshitz constants with respect to different blocks of variables. To find an $\varepsilon$-approximate saddle point, \cite{zhang2019lower} recently showed that any first-order algorithm with the linear span assumption requires at least  
\begin{equation}\label{eq:general_lowerbound}
\Omega\left(\Big(\sqrt{\frac{L_x}{\mu_x} + \frac{L_{xy}^2}{\mu_x\mu_y}+\frac{L_y}{\mu_y}}\Big) \log\Big(\frac{1}{\varepsilon}\Big)\right)
\end{equation}
calls to a gradient oracle for $\phi(x,y)$. Notably, the lower iteration complexity bound applies to even the class of bilinearly coupled quadratic minimax problems, which was used to construct the hard instance. 

In the special parameter regime when $L_x=L_y=L_{xy}:=L$ and $\mu_x=\mu_y:=\mu$, this lower bound is matched by several popular algorithms including the Mirror Prox algorithm \cite{nemirovski2004prox,mokhtari2020unified}, the extra-gradient methods \cite{korpelevich1976extragradient,mokhtari2020unified,gidel2018variational} and the accelerated dual extrapolation \cite{nesterov2018lectures}, with iteration complexity of $O((L/\mu)\log(1/\varepsilon))$. %

However, in the general parameter regime, despite several recent attempts~\cite{cohen2021relative,lin2020near,wang2020improved,zhang2021complexity}, no known algorithms have yet exactly matched the lower bound.  
For instance, the algorithm in \cite{lin2020near} achieves an upper complexity bound of $\tilde O((\cL/\sqrt{\mu_x\mu_y})\log^3(1/\varepsilon))$. One of the best-known results is obtained in  \cite{wang2020improved}, that gives the complexity of 
 $\tilde O(\sqrt{(L_x/\mu_x)+(\cL\,L_{xy}/\mu_x\mu_y)+(L_y/\mu_y)}\log(1/\varepsilon))$,  
where $\tilde O$ hides a polylogarithmic factor in problem parameters and $\cL=\max\{L_x,L_{xy},L_y\}$.
These advances all rely on carefully designed multi-loop algorithms.

We  close this gap for a class of SC-SC minimax problems with bilinear coupling (Bi-SC-SC). Specifically, we consider problems of the general form: 
\begin{equation}\label{eq:bilinear-minimax}
\min_{x \in \cX} \max_{y \in \cY} \; [\phi(x, y) = f(x) + \Ip{y}{Ax} - h(y)],
\end{equation}
where $f(x)$ is $L_x$-smooth and $\mu_x$-strongly convex, $h(y)$ is $L_y$-smooth and $\mu_y$-strongly convex, $\cX$ and $\cY$ are closed convex sets. We assume access to first-order gradient oracles of $f$ and $g$ as well as the matrix $A$.  Note that the lower bound in~\eqref{eq:general_lowerbound} also holds for this class of problems. This class of problems by itself has found numerous applications in machine learning, as detailed in Section~\ref{sec:applications}.

The main challenge in designing an optimal algorithm is that the objective consists of two different classes of functions: smooth convex terms $f$ and $h$, and bilinear coupling $\Ip{y}{Ax}$. These two classes   are traditionally optimized using conceptually different algorithms. On one hand,  accelerated gradient methods (AGD)~\cite{nesterov2018lectures} are optimal at solving smooth strongly convex problems like $\min_x f(x)$ or $\min_y h(y)$.  On the other hand, bilinear problems of the form, $\min_x \max_y \Ip{y}{Ax}$ or the like (with additional proximal-friendly terms), are optimally solved using a seemingly different class of algorithms such as primal-dual methods; see e.g., ~\cite{chambolle2016ergodic,chen1997convergence,bauschke2011convex,chen2014optimal,he2016accelerated}, just to name a few. Such a conceptual difference makes it hard to design an algorithm that achieves optimal dependence on the smoothness and strong convexity parameters of each of the three terms in the objective.

\begin{table*}[t]
	\centering
	\renewcommand{\arraystretch}{1.4}
	\caption{First-order gradient oracle complexity %
	comparisons for bilinearly-coupled smooth minimax problems. %
	We define  $\cL := \max(L_x,\|A\|, L_y)$.  
	$^\clubsuit$Smoothing  applies  a Bi-SC-SC method to the smoothed problem: $\min_x \max_y \phi(x,y)+\Ord(\varepsilon)\|x\|^2$ which makes the originally Bi-C-SC problem $\mu_x=\epsilon$-strongly convex. $^\diamondsuit$Each term of this tight lower bound is implicitly implied by existing lower-bounds for special cases of the Bi-C-SC problem.
	}
	\begin{tabular}{ l c l }
	\hline 
	\textbf{Method} & \textbf{\# Loops} & \textbf{Complexity to reach $\varepsilon$ primal-dual gap} 
	\\ \hline
	\textit{Strongly-Convex--Strongly-Concave (Bi-SC-SC)} \\
	\hline	
	\begin{tabular}{l}
	\rule{0pt}{15pt}
		MP/EG, OGDA 
        {\small\cite{mokhtari2020unified}}, DE {\small \cite{nesterov2006solving}} \\
        \rule{0pt}{15pt}
		MP Bal. {\small(Appendix \ref{sec:expt_details})}, MP RL { {\small\cite{cohen2021relative}}}\\
		\rule{0pt}{17pt}
    	Proximal Best Response  { \small\cite{wang2020improved}}\\
    	\rule{0pt}{15pt}
		DIPPA { \small\cite{xie2021dippa}}\\
		\rule{0pt}{17pt}
		{\bf Lifted PD} { \small (Theorem~\ref{thm:smooth_scsc_primal_dual_informal})}\\
		\rule{0pt}{17pt}
		Lower bound \cite{zhang2019lower}\\
	\end{tabular}
	&
	\begin{tabular}{l}
	\rule{0pt}{15pt}	Single \\
	\rule{0pt}{15pt}	Single \\
    \rule{0pt}{17pt}	Multi \\
	\rule{0pt}{15pt}	Multi \\
	\rule{0pt}{17pt}	Single \\
	\rule{0pt}{17pt}	N/A \\
	\end{tabular}
	&	
	\begin{tabular}{l}
		\rule{0pt}{15pt} 
		$\Ord\big({\frac{L_x + \|A\| + L_y}{\min({\mu_x, \mu_y})}}\big) \log\big(\frac1{\varepsilon}\big)$ 
		\\
				\rule{0pt}{15pt} 
		$\Ord\big({\frac{L_x}{\mu_x}} + \frac{\|A\|}{\sqrt{\mu_x \mu_y}} + {\frac{L_y}{\mu_y}} \big)  \log\big(\frac1{\varepsilon}\big)$ \\
		\rule{0pt}{17pt}
		$\widetilde{\Ord}\big(\sqrt{\frac{L_x}{\mu_x}} + \sqrt{\frac{\|A\| \cL}{\mu_x \mu_y}} + \sqrt{\frac{L_y}{\mu_y}}\big)  \log\big(\frac1{\varepsilon}\big)$ \\
		\rule{0pt}{15pt}
		$\widetilde{\Ord}\big(\,(\frac{L_x^2 L_y}{\mu_x^2 \mu_y})^{\frac14} + \frac{\|A\|}{\sqrt{\mu_x \mu_y}} + (\frac{L_y^2 L_x}{\mu_y^2 \mu_x})^{\frac14}\,\big)  \log\big(\frac1{\varepsilon}\big)$ \\
		\rule{0pt}{17pt}
		$\Ord\big(\sqrt{\frac{L_x}{\mu_x}} + \frac{\|A\|}{\sqrt{\mu_x \mu_y}} + \sqrt{\frac{L_y}{\mu_y}}\big)  \log\big(\frac1{\varepsilon}\big)$ \\
		\rule{0pt}{17pt}
		$\Omega\big(\sqrt{\frac{L_x}{\mu_x}} + \frac{\|A\|}{\sqrt{\mu_x \mu_y}} + \sqrt{\frac{L_y}{\mu_y}}\big)  \log\big(\frac1{\varepsilon}\big)$ \\
	\end{tabular}
	\\\hline
	\textit{Convex--Strongly-Concave (Bi-C-SC)} \\
	\hline	
	\begin{tabular}{l} 
	\rule{0pt}{15pt}
		MP/EG, OGDA 
        {\small\cite{mokhtari2020unified}}, DE {\small \cite{nesterov2007dual}} \\
		\rule{0pt}{17pt}
		PDHG-type   
		{\small\cite{zhao2019optimal}} \\	
		\rule{0pt}{17pt}
		DIAG  {\small\cite{thekumparampil2019efficient}} \\
		\rule{0pt}{17pt}
		{\bf Lifted PD} { \small (Theorem~\ref{thm:smooth_csc_primal_dual_informal})}\\
		\rule{0pt}{17pt}
		{\bf Lifted PD + Smoothing}$^\clubsuit$ { \small (Remark~\ref{rem:smooth_scsc_primal_dual_plus_smoothing})}\\
		\rule{0pt}{17pt}
		Lower bound$^\diamondsuit$ \cite{ouyang2021lower,nesterov2018lectures} \\
	\end{tabular}
	& 
	\begin{tabular}{l}
		\rule{0pt}{15pt} Single \\
		\rule{0pt}{17pt} Multi \\
		\rule{0pt}{17pt} Multi \\
		\rule{0pt}{17pt} Single \\
		\rule{0pt}{17pt} Single \\
		\rule{0pt}{17pt} N/A \\
	\end{tabular}
	&
	\begin{tabular}{l}
	\rule{0pt}{15pt}
		$\Ord\big(\frac{L_x + \|A\| + L_y }{\varepsilon}\big)$
		\\
		\rule{0pt}{17pt}
		$\Ord \big({\frac{L_x}{\varepsilon}} + \frac{\|A\|}{\sqrt{\mu_y \varepsilon}} + \sqrt{\frac{L_y}{\mu_y}}\log(\frac{1}{\varepsilon})    \big)$\\		
		\rule{0pt}{17pt}
		$\Ord\big( \sqrt{\frac{L_y}{\mu_y}} \big(\sqrt{\frac{L_x}{\varepsilon}} + \frac{\|A\|}{\sqrt{\mu_y \varepsilon}} \big) \big) \log^2\big(\frac1{\varepsilon}\big)$
		\\		
		\rule{0pt}{17pt}
		$\Ord\big(\sqrt{\frac{L_x}{\varepsilon}} + \frac{\|A\|}{\sqrt{\mu_y \varepsilon}} + \sqrt{\frac{L_y}{\varepsilon}}\big)$
		\\
		\rule{0pt}{17pt}
		$\Ord\big(\sqrt{\frac{L_x}{\varepsilon}} + \frac{\|A\|}{\sqrt{\mu_y \varepsilon}} + \sqrt{\frac{L_y}{\mu_y}}\big)  \log\big(\frac1{\varepsilon}\big)$ 
		\\	
		\rule{0pt}{17pt}
		$\Omega \big(\sqrt{\frac{L_x}{\varepsilon}} + \frac{\|A\|}{\sqrt{\mu_y \varepsilon}} + \sqrt{\frac{L_y}{\mu_y}}\log(\frac{1}{\varepsilon})    \big)$\\
	\end{tabular}
	\\\hline
\end{tabular}
\label{tab:complexity} 
\end{table*}

\subsection{Our Contributions}
We introduce a new algorithm that reconciles these different components by lifting the objective to an extended saddle point formulation.  Our key idea hinges on the recent interpretation~\cite{lan2018optimal} of accelerated gradient descent for convex minimization as a variant of primal-dual method for an equivalent minimax problem. Based on the reformulation, we can handle both the smooth terms and bilinear coupling term under the same umbrella of primal-dual method. 
We make the following key contributions.
\begin{itemize}
    \item We provide the first optimal algorithm for the class of bilinearly coupled SC-SC minimax problems, called the \emph{lifted primal-dual (LPD) method}, achieving the iteration complexity of 
    $\Ord\big((\sqrt{L_x/\mu_x}+\|A\|/\sqrt{\mu_x\mu_y}+\sqrt{L_y/\mu_y})\log(1/\varepsilon)\big)$ (Theorem \ref{thm:smooth_scsc_primal_dual_informal}),
     tightly matching the lower bound. The LPD method is also \emph{single-loop}, using only one gradient oracle call per iteration, which is more desirable in practice. 

    \item  For bilinearly coupled \textit{convex-strongly-concave} (Bi-C-SC) minimax problems where $f$ is only convex, namely, $\mu_x = 0$, 
    we can apply the LPD method to a {\em smoothed} objective $\phi(x,y) + \lambda \varepsilon \|x\|^2$, which transforms the objective into a SC-SC one \cite{nesterov2005smooth}. The LPD method is the first  to achieve optimal complexity up to logarithmic factors in this setting (Remark \ref{rem:smooth_scsc_primal_dual_plus_smoothing}) as shown in Table~\ref{tab:complexity}. 
    However, smoothing might not be desirable in practice (see Section~\ref{sec:analysis}).  %
    To this end, 
    we design a direct algorithm by selecting appropriate stepsizes in LPD.  This  achieves an iteration complexity  %
    that is suboptimal but the best among those not using smoothing (Theorem~\ref{thm:smooth_csc_primal_dual_informal}).  %
    
\end{itemize}

Detailed comparisons with existing algorithms are presented in Table~\ref{tab:complexity}.

\subsection{Related Work}
\label{sec:related}
Below we highlight key distinctions of our work to the most closely related literature. Our list of related work is by no means comprehensive. 
There exists optimal algorithm for the case when both $f$ and $h$ are just convex ($\mu_x$$\,=\,$$\mu_y$$\,=\,$$0$) \cite{chen2014optimal,chen2017accelerated}. However, when either of $f$ or $h$ is strongly convex it is not readily clear how to optimally solve the problem.

\paragraph{Bilinear coupling with simple terms.} Existing work on bilinearly coupled minimax problems primarily focuses on the case when $f$ and/or $h$ are proximal-friendly, i.e., it is easy to compute the proximal operator.  If both $f$ and $h$ are proximal-friendly and strongly convex, then the primal-dual method~\cite{chambolle2016ergodic} and accelerated forward-backward algorithm~\cite{palaniappan2016stochastic} already achieve  the optimal rate $\Ord( (\|A\| /\sqrt{\mu_x \mu_y})\log(1/\varepsilon))$ \cite{xie2021dippa}. 
If only $h$ is proximal-friendly and $f$ is smooth, but both are strongly convex, \cite{chambolle2016ergodic} provides a linearly convergent but sub-optimal algorithm.
Our work differs from this line of results as we don't require computing the proximal operators of neither $f$ nor $h$, but instead only their gradients. 
One exception is DIPPA, a complex multi-loop algorithm introduced in 
\cite{xie2021dippa}, which achieves  
$\Ord( (((L_x L_y/\mu_x \mu_y)(L_x/\mu_x + L_y/\mu_y))^{1/4} + \|A\|/\sqrt{\mu_x\mu_y})\log(1/\varepsilon))$ complexity under the same setting (Bi-SC-SC) as the one we study. Additionally, for the special case of quadratic Bi-SC-SC problem, \cite{wang2020improved} provides a recursive multi-loop algorithm which achieves a sub-optimal iteration complexity of $\Ord( (\sqrt{L_x/\mu_x} + \|A\|/\sqrt{\mu_x \mu_y} + \sqrt{L_y/\mu_y} ) \, (\cL/\mu_x \mu_y)^{o(1)} \, \log(1/\varepsilon ))$, where $\cL=\max(L_x, L_{xy}=\|A\|, L_y)$.

\paragraph{Beyond bilinear coupling.} Beyond bilinear coupling, most existing work either treat the  objective as a whole or consider special couplings. 
In the SC-SC setting with a general coupling \eqref{eq:general_minimax}, there are many algorithms which achieve linear convergence; one of the first such algorithm is the Extragradient (EG) method  \cite{korpelevich1976extragradient,tseng1995linear}.  Here, Gradient  Descent Ascent (GDA) achieves an iteration complexity of $\Ord(\kappa_{\max}^2 \log(1/\varepsilon))$ (see e.g., Chapter 12 in \cite{facchinei2007finite}), while Mirror-Prox (MP/EG) \cite{nemirovski2004prox}, Dual Extrapolation (DE) \cite{nesterov2007dual,nesterov2006solving}, and Optimistic Gradient Descent Ascent (OGDA) \cite{daskalakis2017training,gidel2018variational,mokhtari2020unified} achieve an iteration complexity of $\Ord(\kappa_{\max} \log(1/\varepsilon))$, where $\kappa_{\max} = \cL/\min(\mu_x, \mu_y)$. Further, the above complexities can be improved (replacing $\kappa_{\max}$ with $L_x/\mu_x + L_{xy}/\sqrt{\mu_x \mu_y} + L_y/\mu_y$) with proper balancing of distance functions  (see Appendix \ref{sec:expt_details}). This improved complexity can also be attained by a modified MP via relative Lipschitzness (we refer to as MP RL)\cite{cohen2021relative}. Two multi-loop algorithms: Minimax-APPA \cite{lin2020near} and Catalyst-type method \cite{alkousa2020accelerated}, which are based on accelerated minimization methods, achieve the iteration complexities of $\widetilde{\Ord}( (\cL/\sqrt{\mu_x \mu_y}) \log^3(1/\varepsilon))$ and $\widetilde{\Ord}(\sqrt{{L_y}/{\mu_y}}(\sqrt{L_x/\mu_x} + L_{xy}/\sqrt{\mu_x\mu_y}) \log^2(1/\varepsilon))$, respectively.
The current state-of-the-art complexity for general coupling is achieved by a multi-loop algorithm \cite{wang2020improved} (see Table \ref{tab:complexity}), but there is a gap to the known lower bound of \eqref{eq:general_lowerbound} as discussed in detail in Section~\ref{sec:intro}.

\paragraph{Convex-Strongly-Concave minimax problems.} 
As an example of special couplings, if the coupling is linear only in $x$ but nonlinear in $y$, and $f$ is proximal-friendly and convex and $h$ is smooth and strongly convex,  \cite{juditsky2011first,hamedani2021primal} achieve $\Ord((\|A\|/\sqrt{\mu_y \varepsilon} + (L_y/\mu_y)) \log(1/\varepsilon))$ and $\Ord((\|A\|+ L_y )/\sqrt{\mu_y \varepsilon})$ complexities, respectively, 
under our setting. 
In the same setting as these works, \cite{zhao2019optimal} provides a PDHG-type algorithm which works even when the coupling is non-linear instead of our bilinear one $\Ip{y}{Ax}$. This leads to a sub-optimal (in $\varepsilon$) complexity of $\Ord({L_x/\varepsilon} + \|A\|/\sqrt{\mu_y \varepsilon} + \sqrt{L_y/\mu_y} \log(1/\varepsilon))$ for our Bi-C-SC setting.
For minimax problems that are not necessarily SC-SC, recent work~\cite{du2019linear,yang2020global,azizian2020tight} show that linear convergence can still be achieved under additional assumptions. 
A recent work~\cite{thekumparampil2019efficient} discussed general convex-concave minimax problems with one-sided strong convexity (i.e., C-SC setting) and obtained a $\Ord(1/\sqrt{\varepsilon})$ complexity (see Table \ref{tab:complexity}). 
In principle, any of the known algorithms \cite{mokhtari2020unified,lin2020near,yang2020catalyst,wang2020improved,xie2021dippa} for the Bi-SC-SC setting \eqref{eq:bilinear-minimax} can be applied to solve Bi-C-SC problem after a smoothing transformation \cite{nesterov2005smooth}. However, due to their sub-optimality in the original Bi-SC-SC case itself, their complexities for C-SC case are sub-optimal as well.
We omit discussions of other minimax optimization settings
as they are less relevant.

\subsection{Notations}
We use $\Ip{x}{y}$ to denote the inner product between vectors $x$ and $y$, and $\|x\|$ to denote Euclidean norm of $x$. For a convex set $\cX$, $\cP_\cX(\cdot)$ denotes its projection operator. We use the standard big-O $\Ord$ and $\Omega$ notations. 
\textit{Iteration complexity or (gradient) complexity} of an algorithm is the number of iterations or gradients used by it find an \textit{$\varepsilon$-approximate saddle point} $(\hat{x}, \hat{y})$, which means that its primal-dual gap $\max_y \phi(\hat{x},y) - \min_x \phi(x,\hat{y}) \leq \varepsilon$.
Standard definitions of $L$-smoothness, $\mu$-strong convexity, Fenchel/convex conjugate, Bregman divergence and its distance generating function, and proximal operators are given in Appendix \ref{sec:prelim_details}.

\section{Problem Setting and Applications}
\label{sec:applications}

We are mainly interested in the \emph{bilinearly coupled strongly-convex-strongly-concave (Bi-SC-SC) minimax problem} of the form (\ref{eq:bilinear-minimax}). 
Throughout, we make the following assumption. 
\begin{assumption}\label{assume:ldp_main_assumptions}
$f$ is $L_x$-smooth and $\mu_x$-strongly convex, and $h$ is $L_y$-smooth and $\mu_y$-strongly convex on the entire Euclidean space. 
\end{assumption}
In addition, we assume that sets $\cX$ and $\cY$ are closed convex and the projection onto these sets is easily computable. Functions $f$ and $h$ have well defined gradient on $\cX$ and $\cY$ and they can be accessed through gradient oracles. 
Two distinctions that differ from most existing work are ($i$) no requirement on computing proximal operator of either $f$ or $h$, and ($ii$) the linear coupling term. 

This type of problems find numerous applications in machine learning. Below we  list only a few. 
\subsection{Quadratic Minimax Problems}

Quadratic minimax problems are fundamental problems which arise in numerical analyses \cite{bai2003hermitian,benzi2005numerical,bai2009optimal,wang2020improved}, optimal control problems \cite{rockafellar1987linear,liu2015projection}, constrained matrix games~\cite{xie2021dippa}. They also appear naturally  when solving subspace proximal sub-problems of Sequential Subspace Optimization for quadratic saddle-point problems \cite{choukroun2020primal}, and when solving sub-problems of minimax (cubic regularized) Newton method \cite{schafer2020competitive,zhang2020newton,huang2020cubic}. 
Here $f(x) = x^T B x$ and $h(y) = y^\top C y$  correspond  to positive definite matrices $B\succ 0$ and $C\succ0$. Thus the minimax objective is quadratic in $x$ and $y$:
\begin{align}\label{eq:quadratic_problem}
\phi(x,y) = x^\top B x + y^\top A x  - y^T C y.
\end{align}
Despite their simplicity, quadratic minimax problems are not trivial to solve \cite{zhang2021don}. Further, 
even nonconvex-nonconcave minimax problems can be Bi-SC-SC near a strict local saddle point   \cite{azizian2020accelerating}. 

\subsection{Robust Least Squares} 
Consider the robust least squares problem \cite{el1997robust,yang2020global} with a coefficient matrix $A$ and noisy vector $y$, where $y$ is corrupted by a deterministic perturbation $\delta$ of a bounded norm $\rho$: 
\begin{align*}
\min_x \max_{\delta: \|\delta\| \leq \rho} \|Ax-y\|^2 \,, \text{ where } \delta = y - y_0.
\end{align*}
The corresponding penalized version of the objective is a Bi-SC-SC minimax problem:
\begin{align*}
\min_{x} \max_{y} \phi(x,y) := \|Ax-y\|^2 - \lambda \|y-y_0\|^2\;.
\end{align*}
Selecting $\lambda>1$, we get a Bi-SC-SC problem.

\subsection{Policy Evaluation}

Bi-SC-SC arise in policy evaluation problem in reinforcement learning \cite{du2017stochastic,du2019linear} when finding minimum the mean squared projected Bellman error (MSPBE). Empirical estimator of minimum MSPBE has the form:
\begin{align}
\arg\min_\theta \frac12 \|A\theta - b\|^2_{C^{-1}} + \frac{\rho}2 \|\theta\|^2, 
\end{align}
where $A$, $b$, $C$ are defined as follows. Suppose we have a trace of $n$ tuples of current-state $s_t$, action $a_t$, next-state $s_{t+1}$, and reward $r_t$ under some policy $\pi$ on some MDP. Then we define $b = 1/n \sum_{t}^n r_t \phi_t$
\begin{align*}
A = \frac1n \sum_{t}^n \phi_t(\phi_t - \gamma \phi_{t+1})^\top \,,\, \text{ and } C = \frac1n \sum_{t}^n \phi_t \phi_t^\top
\end{align*}
where $\phi_t$ is the feature of state $s_t$, $\gamma$ is the discount factor. In practice, inverting $C$ can be computationally costly. Therefore, one may resort to solving the following minimax reformulation, eliminating the need for matrix inversion.
\begin{align}
\min_\theta \max_w  \frac{\rho}2 \|\theta\|^2 - w^\top A \theta - (\frac12 \|w\|^2_C - w^\top b) 
\label{eq:policy_eval_problem}
\end{align}
This is a Bi-SC-SC problem if $C$ is positive definite.

Note that, all these problems becomes a Bi-Convex--Strongly-Concave (Bi-C-SC) or Bi-Strongly-Convex--Concave (Bi-SC-C) problem, if the Hessian of convex quadratic of the primal or dual variables, respectively, becomes positive semi-definite.

\section{Building Blocks for our Method}
\label{sec:building_blocks}

We present known results that serve as the intuition behind the design of Algo.~\ref{algo:lpd}.
We first revisit the primal-dual method~\cite{chambolle2016ergodic},  which is originally designed to solve bilinearly coupled minimax problems with simple terms whose proximal operators are easy to compute.  We also discuss how it can be used to optimize smooth convex objectives with accelerated convergence. 

\subsection{Primal-Dual method \texorpdfstring{\cite{chambolle2016ergodic}}{[CP16]}} %
\label{sec:31}

Consider the bilinearly coupled minimax problem: 
\begin{align}\label{eq:scsc_prox_problem}
\min_x \max_y\;  F(x) + \Ip{y}{Ax} - H(y)
\end{align} 
with a unique solution $z^*=(x^*,y^*)$. Additionally, let $r$ and $s$ be $1$-strongly convex distance generating functions (d.g.f.) that induce Bregman divergences $V^{r}_{x_0}(x)$ and $V^{s}_{y_0}(y)$. Then, we assume that $F$ and $H$ are relatively $\mu_x$- and $\mu_y$-strongly convex with respect to $V^{r}_{x_0}(x)$ and $V^{s}_{y_0}(y)$, respectively. We also assume the access to their Bregman proximal operators, 
with respect to the corresponding divergences. 

The PD method can be viewed as an approximation of proximal point method (PPM) \cite{rockafellar1976monotone}. 
We emphasize this connection as the analyses of our main results closely follow that of PPM (Lemma \ref{lem:ppm_rule}). Readers who are familiar with this connection may skip to after Lemma \ref{lem:ppm_rule}.  The PPM updating rule is as follows:
\begin{equation*}
\begin{aligned}
&(x_{k+1},y_{k+1}) \,=\, \arg\min_{x} \, \arg\max_{y} \\
&\left\{\frac{1}{\eta_x}V^{r}_{x_k}(x) + F(x) + 
\Ip{y}{A x} - H(x) -  \frac{1}{\eta_y}V^{s}_{y_k}(y)\right\}.
\end{aligned}
\end{equation*}
This is equivalent to the implicit update rule
\begin{equation}
\left\{
\begin{aligned}
x_{k+1} &= \arg\min_{x} \Ip{A^\top y_{k+1}}{x} + \frac{1}{\eta_x}V^{r}_{x_k}(x) + F(x) \;\;\;\;\;\;\;\;\; 
\nonumber 
\\
y_{k+1} &= \arg\min_{y} -\Ip{A x_{k+1}}{y} + \frac{1}{\eta_y}V^{s}_{y_k}(y) + H(y).
\end{aligned}
\right.
\label{eq:ppm_rule}
\end{equation}
This is a conceptual rule and not an implementable one because finding $x_{k+1}$ requires the gradient at $y_{k+1}$ and vice versa. It is easy to prove that iterates of PPM linearly converges to the solution of \eqref{eq:scsc_prox_problem}. We provide a proof in Appendix \ref{sec:ppm_proof} for completeness.
\begin{lemma}\label{lem:ppm_rule}
	The iterates of the PPM  for the problem \eqref{eq:scsc_prox_problem} satisfy $(\|x^*-x_{K}\|^2/\eta_x + \|y^*-y_{K}\|^2/\eta_y) \leq 2 \exp(-K/(1+\kappa)) (V^{r}_{x_0}(x^*)/\eta_x + V^{s}_{y_0}(y^*)/\eta_y)$ for all $K \geq 0$, where $\kappa = 1/\min(\mu_x \eta_x, \mu_y \eta_y)$.
\end{lemma}

PD method is the following approximation of PPM: 
\begin{equation}
\left\{
\begin{aligned}
\ty_{k+1} &= y_{k} + \theta (y_{k} - y_{k-1})  \\
x_{k+1} &= \arg\min_{x} \Ip{A^\top \ty_{k+1}}{x} + \frac{1}{\eta_x}V^{r}_{x_k}(x) + F(x)  \\
y_{k+1} &= \arg\min_{y} -\Ip{A x_{k+1}}{y} + \frac{1}{\eta_y}V^{s}_{y_k}(y) + H(x)
\end{aligned}
\right.
\label{eq:pd_rule}
\end{equation}
where $\theta = 1/\gamma$ and $\gamma \leq 1 + \min(\mu_x \eta_x, \mu_y \eta_y)$. Different from PPM, the PD method uses a pseudo-gradient $A^\top \ty_{k+1}$ computed at the extrapolated $\ty_{k+1}$, instead of the actual gradient $A^\top y_{k+1}$ at $y_{k+1}$, to update $x_{k+1}$. This approximation leads to an implementable algorithm with the same linear convergence as PPM. 
\begin{theorem}[\cite{chambolle2016ergodic}]\label{thm:pd_rule}
	If 
	$\sqrt{\mu_x/\mu_y}\eta_x$$\,=\,$$\sqrt{\mu_y/\mu_x} \eta_y$$\,=\,$$1/2\|A\|$ and $y_{-1} = y_{0}$,
	then the iterates of the PD update rule \eqref{eq:pd_rule} satisfy the same conclusion as Lemma \ref{lem:ppm_rule}, with $\kappa$$\,=\,$$2\|A\|/\sqrt{\mu_x \mu_y}$.
\end{theorem}
For completeness, we provide  a proof in Appendix \ref{sec:pd_proof}.
The PD method obtains the optimal iteration complexity of $\Ord(\|A\| \log(1/\varepsilon)/\sqrt{\mu_x \mu_y})$ \cite{xie2020lower,han2021lower}. We also note that a particular version of the PD method can also be interpreted as an \textit{exact} PPM update using the the Bregman divergence corresponding to the bilinear operator $A$ \cite{he2012convergence}.

\subsection{Accelerated Convex Minimization \texorpdfstring{\cite{lan2018optimal}}{[LZ18]}} %
\label{sec:32}

In this section, we illustrate that the PD method can also be deployed to optimally solve strongly convex and smooth minimization problems. Consider the problem: $\min_x f(x)$, where function $f$ is $L$-smooth and $\mu$-strongly convex, and  the optimal solution is $x^*$.

First, we reformulate it into the following minimax problem by introducing a dual variable $u$ and \textit{lifting} it into a larger variable space $(x,u)$:
\begin{align}
\label{eq:sc_agd_minimax_problem}
\min_x \Big[\,f(x) = \max_u\; \frac{\mu}{2} \|x\|^2 + \Ip{x}{u} - \uf^*(u)\,\Big]\;,
\end{align}
where 
$$\uf(x)=f(x)-\frac{\mu}{2}\|x\|^2, \,\,\,\,\uf^*(u)=\max_x \Ip{u}{x}-\uf(x).$$
Here $\uf^*$ is the Fenchel/convex conjugate of $\uf$. 
Following  the definition, we have $\uf(x)$ is $(L-\mu)$-smooth and convex. A proof is provided in Appendix \ref{sec:sc_minux_quad_pf} for completeness.
Then its dual $\uf^*$ is $(L-\mu)^{-1}$ strongly convex with respect to Euclidean norm \cite{beck2017first}.

Notice that this new minimax problem is in the form \eqref{eq:scsc_prox_problem}, with bilinear coupling matrix $A=\Id$, $\mu$-strongly convex function $F(x)=\frac{\mu}{2} \|x\|^2 $, and 
$1$-relatively strongly convex function $H(u)$$\,=\,$$\uf^*(u)$ with respect to the Bregman divergence $V^{\uf^*}_{u_{k}}(u)$ generated by $\uf^*$ itself.

Then, if we instantiate the PD update rule \eqref{eq:pd_rule} for this problem, we obtain the updates:
\begin{equation}
\left\{
\begin{aligned}
\tu_{k+1} &= u_{k} + \theta (u_{k} - u_{k-1}) \\
x_{k+1} &= \arg\min_{x} \Ip{\tilde{u}_{k+1}}{x} +\frac{\mu}{2} \|x\|^2+\frac{1}{2\eta_x}\|x-x_k\|^2  \\
u_{k+1} &= \arg\min_{u} -\Ip{x_{k+1}}{u} + \uf^*(u) + \frac1{\eta_u} V^{\uf^*}_{u_{k}}(u)
\end{aligned}
\right.
\label{eq:sc_agd_rule}
\end{equation}

\begin{corollary}[of Theorem \ref{thm:pd_rule}] \label{cor:sc_agd}
	Let $u_{-1} = u_{0} = \nabla f(x_0)$. Then the iterates of the PD update rule \eqref{eq:sc_agd_rule} with stepsizes $\eta_x$$\,=\,$$1/\sqrt{\mu (L - \mu)}$, $\eta_u$$\,=\,$$\sqrt{\mu/(L - \mu)}$ and $\theta$$\,=\,$$(1+\sqrt{\mu/(L-\mu)})^{-1}$, for problem \eqref{eq:sc_agd_minimax_problem} satisfies $\|x^* - x_K\|^2 \leq \Ord(\exp(-K/2\sqrt{\kappa-1}) \| x^* - x_0\|^2)$ for all $K \geq 0$, where $\kappa = L/\mu$.
\end{corollary}

A proof is in Appendix \ref{sec:sc_agd_pf}. Note that this matches the optimal convergence rate achieved by Nesterov's accelerated gradient descent (AGD) methods \cite{nesterov2018lectures}.

Finally, we show that Bregman proximal update rule in \eqref{eq:sc_agd_rule} admits an elegant implementation based on the gradients, which resembles AGD.
\begin{lemma} \label{lem:sc_agd_update}
	For problem \eqref{eq:sc_agd_minimax_problem}, iterates $x_{k}$ of the PD update rule \eqref{eq:sc_agd_rule} are the same as the iterates $x_{k}$ of the following update rule when $(u_{-1},u_{0})$$\,=\,$$(\nabla f(\ux_{-1}),\nabla f(\ux_{0}))$.
	\begin{equation}
	\left\{
	\begin{aligned}
	\tnabla_{k+1} &= \nabla \uf(\ux_{k}) + \theta (\nabla \uf(\ux_{k}) - \nabla \uf(\ux_{k-1}))  \\
	x_{k+1} &= (x_k - \eta_x \tnabla_{k+1})/(1+\eta_x \mu) \\
	\ux_{k+1} &= ({\ux_{k} + \eta_u x_{k+1}})/({1+\eta_u}) 
	\end{aligned}
	\right.\label{eq:sc_equiv_agd_rule} 
	\end{equation}
\end{lemma}
A proof is in Appendix \ref{sec:sc_agd_update_pf}. 
This close connection between the PD method and AGD is first identified in~\cite{lan2018optimal}. The above analysis based on the primal-dual interpretation is conceptually much simpler than the more opaque estimate sequence \cite{nesterov2018lectures} or Lyapunov-based \cite{lan2012optimal} analyses of AGD.
Note that the above update rule is slightly different from the one used in~\cite{lan2018optimal}. The latter first use extrapolated primal iterate $\tx_{k+1}$ to update dual iterate $u_{k+1}$, whereas we use extrapolated dual iterate $\tu_{k+1}$ to update the primal iterate $x_{k+1}$.

\section{Lifted Primal-Dual method}
The previous section indicates that both bilinear minimax problems and smooth strongly convex minimization problems can be optimally solved using the same PD method after appropriate reformulation. Naturally, this  suggests that the PD method has the potential to solve the Bi-SC-SC problem of our interest:
\begin{equation}\label{eq:scsc_problem}
\min_{x \in \cX} \max_{y \in \cY} \; [\phi(x, y) = f(x) + \Ip{y}{Ax} - h(y)],
\end{equation}
which at consists of a bilinear term $\Ip{y}{Ax}$ and two smooth strongly-convex functions $f$ and $h$. 

Our strategy to solve \eqref{eq:scsc_problem} is to first transform the objective into a form where the proximal operators are easy to compute and then solve this new objective using the PD method. Introducing dual variables $u$ and $v$ for $f$ and $h$, respectively, 
the Bi-SC-SC problem can be equivalently reformulated (or {\em lifted}) as
\begin{align}
\min_{x \in \cX, v}\; \max_{y \in \cY, u}\; \Phi(x,y;u,v) \,, \text{ where } \label{eq:scsc_reform_problem}
\end{align}
\begin{align}
&\Phi(x,y;u,v) := \big[-\uf^*(u) + \Ip{u}{x} + ({\mu_x}/2) \|x\|^2\big] \nonumber \\
&+ \Ip{y}{Ax} - \big[({\mu_y}/2) \|y\|^2 + \Ip{v}{y} - \uh^*(v)\big]
\,,
\label{eq:scsc_reform_problem_obj}
\end{align}
\begin{align}
\uf^*(u) &:= 
\max_{x}
\Ip{u}{x} - [\uf := f(x) - ({\mu_x}/2) \|x\|^2] \,, \text{ and }\nonumber \\
\uh^*(v) &:= 
\max_{y}
\Ip{v}{y} - [\uh := h(x) - ({\mu_y}/2) \|y\|^2].
\end{align}
By Fenchel duality, it follows that $\phi(x,y)=\min_{v}\max_{u}\Phi(x,y;u,v)$ (Lemma \ref{lem:fenchel_props}(c)).
Note that both $\uf^*$ and $\uh^*$ are strongly convex.
Intriguingly, the first three terms, the middle three terms, and  the last three terms in \eqref{eq:scsc_reform_problem} are all of the form \eqref{eq:scsc_prox_problem} amenable for Primal-Dual approach. To this end, we introduce  the following  the PD update to each of the four variables with their respective stepsizes, Bregman divergences, and extrapolation steps.
\begin{align}
&(\tx_{k+1}, \ty_{k+1}) = (1+\theta)(x_{k}, y_{k}) - \theta(x_{k-1}, y_{k-1}) \nonumber \\
&(\tu_{k+1}, \tv_{k+1}) = (1+\theta)(u_{k}, v_{k}) - \theta(u_{k-1}, v_{k-1}) \nonumber \\
&x_{k+1} = \arg\min_{x \in \cX} \Ip{A^\top \ty_{k+1} + \tu_{k+1}}{x} + \nonumber \\
&\;\;\;\;\;\;\;\;\;\;\;\;\;\;\;\;\;\;\;\;\;\;\;\;\; \|x-x_{k}\|^2/{2\eta_{x}} + {\mu_x}\|x\|^2/2 \nonumber\\
&y_{k+1} = \arg\min_{y \in \cY} -\Ip{A^\top \tx_{k+1} + \tv_{k+1}}{y} + \label{eq:scsc_lpd_rule} \\
&\;\;\;\;\;\;\;\;\;\;\;\;\;\;\;\;\;\;\;\;\;\;\;\;\; \|y-y_{k}\|^2/{2\eta_{y}} + {\mu_y}\|y\|^2/2 \nonumber\\
&u_{k+1} = \arg\min_{u} -\Ip{x_{k+1}}{u} + \uf^*(u) + V^{\uf^*}_{u_k}(u)/{\eta_u} \nonumber \\
&v_{k+1} = \arg\min_{v} -\Ip{y_{k+1}}{v} + \uh^*(v) + V^{\uh^*}_{v_k}(v)/{\eta_v} \nonumber 
\end{align}

We show that the above update rule can be easily implemented using Algorithm \ref{algo:lpd}, which we call the \emph{Lifted Primal-Dual (LPD)} method.
\begin{lemma}[Same as Lemma \ref{lem:scsc_lpd_update_appendix}] \label{lem:scsc_lpd_update}
	For problem \eqref{eq:scsc_reform_problem}, iterates $(x_{k},y_{k})$ of the PD update rule \eqref{eq:scsc_lpd_rule} is the same as the iterates $(x_{k}, y_{k})$ of Algorithm \ref{algo:lpd}, when $(u_{-1}, u_0)$$\,=\,$$(\nabla \uf(\ux_{-1}), \nabla \uf(\ux_0))$,  $(v_{-1}, v_0)$$\,=\,$$(\nabla \uf(\uy_{-1}), \nabla \uf(\uy_0))$, and stepsizes $(\eta_{x,k}, \eta_{y,k}, \eta_{u,k}, \eta_{v,k}, \theta_{k})$ are  invariant to $k$.
\end{lemma}
We omit the proof of the above lemma as it is similar to that of Lemma \ref{lem:sc_agd_update}. 
\begin{algorithm}[t!]
	\SetAlgoLined
	\DontPrintSemicolon
	\SetKwProg{myproc}{Procedure}{}{}
	{\bf Required}: $\cX$, $\cY$, $(f, L_x, \mu_x)$, $(A, \|A\|)$, $(h, L_y, \mu_y)$, %
	$K$, \\
	$\{(\eta_{x,k}, \eta_{y,k}, \eta_{u,k}, \eta_{v,k}, \theta_{k})\}_{k=0}^{K-1}$ \\
	\nl Initialize $(x_{-1}, y_{-1}) = (x_0, y_0) \in \cX \times \cY$ \\
	\nl Set $\uf = f - (\mu_x/2) \|\cdot\|^2$, $\uh = h - (\mu_y/2) \|\cdot\|^2$, \\
	$(\ux_{-1}, \uy_{-1}) = (\ux_0, \uy_0) = (x_0, y_0)$ \\
	\For{$0\leq k\leq K-1$}{ 
		\nl $\tx_{k+1} = x_{k} + \theta_k (x_k - x_{k-1}) \;,\;\;\;\; 
		\nonumber \\
		\ty_{k+1} = y_{k} + \theta_k (y_k - y_{k-1})$, \\
		$\tnabla_{x, {k+1}} = \nabla \uf(\ux_{k}) + \theta_k (\nabla \uf(\ux_{k}) - \nabla \uf(\ux_{k-1})) \;,\;\;\;\; \nonumber \\
		\tnabla_{y, {k+1}} = \nabla \uh(\uy_{k}) + \theta_k (\nabla \uh(\uy_{k}) - \nabla \uh(\uy_{k-1}))$
		\\
		\label{algo_line:extrapolate_lpd} \\
		\nl $x_{k+1} = \cP_\cX((x_k - \eta_{x,k} (A^\top \ty_{k+1} + \tnabla_{x, {k+1}}))/$ \label{algo_line:x_update_lpd} \\
		${\color{white} y_{k+1} = \cP_\cY(}(1+\eta_{x,k}\mu_y))$ \\
		\nl $y_{k+1} = \cP_\cY((y_k + \eta_{y,k} (A \tx_{k+1} - \tnabla_{y, {k+1}}))/$ \label{algo_line:y_update_lpd} \\
		${\color{white} y_{k+1} = \cP_\cY(}(1+\eta_{x,k}\mu_y))$ \\
		\nl $\ux_{k+1} = ({\ux_{k} +\eta_{u,k}\, x_{k+1}})/({1+\eta_{u,k}})$ \;,\;\;\; \\
		\nl $\uy_{k+1} = ({\uy_{k} +\eta_{v,k}\, y_{k+1}})/({1+\eta_{v,k}})$
	}
	\nl \Return $(x_K, y_K, \ux_K, \uy_K)$
	\caption{LPD: Lifted Primal-Dual algorithm}
	\label{algo:lpd}
\end{algorithm}
Note that we update the variables in the order $(x, y) \to (\ux, \uy)$, where variables in tuples are simultaneously updated. However, any update ordering can be shown to achieve similar guarantees as we show, by using appropriate extrapolation steps and stepsize choices.
We extrapolate all the variables and gradients (in step \ref{algo_line:extrapolate_lpd} of Algorithm \ref{algo:lpd}) before the $x$ and $y$ updates (steps \ref{algo_line:x_update_lpd} and \ref{algo_line:y_update_lpd} of Algorithm \ref{algo:lpd}) to make our analysis a bit symmetric, hence simpler. However, depending on the order in which we update each of the variables we may not have to extrapolate all the variables. For example, if we update variables in the order $x \to y \to \ux \to \uy$, we only have to use ($a$) the extrapolated $\widetilde{y}_{k+1}$ and $\widetilde{\nabla}_{x,k+1}$ for updating $x$, and ($b$) the extrapolated $\widetilde{\nabla}_{y,k+1}$ for updating $y$.

\section{Convergence Analysis}
\label{sec:analysis}
Now we provide the main theoretical results.

\textbf{Strongly-Convex--Strongly-Concave Case.}
  LPD achieves the optimal iteration complexity for solving Bi-SC-SC problems  in  \eqref{eq:scsc_problem}.
Define the following condition numbers:
$\kappa_x={L_x}/{\mu_x},$  $\kappa_y={L_y}/{\mu_y}$,  $\kappa_{xy}={\|A\|}/{\sqrt{\mu_x \mu_y}}$, 
and define the meta-condition number:
$
\kappa=\sqrt{\kappa_x-1} + 2\kappa_{xy} + \sqrt{\kappa_y-1}\;.
$
Let $x^*,y^*$ be the optimal solution. For any candidate solution $(x,y)\in\mathcal{X}\times\mathcal{Y}$, we measure the suboptimality with,  
$$
\epsilon(x,y)=\kappa_{xy}(\mu_x\|x-x^*\|^2+\mu_y\|y-y^*\|^2).
$$

\begin{theorem} 
[Informal version of Corollary \ref{cor:smooth_scsc_primal_dual}]
\label{thm:smooth_scsc_primal_dual_informal}
For any $k\geq 0$, set the parameters 
\begin{align}
\gamma =\,  &1 + \kappa^{-1} \,,\, \theta_k = 1/\gamma \,,\, \nonumber \\
\eta_{x,k} &=(\sqrt{\kappa_x-1} + 2\kappa_{xy})^{-1}/\mu_x
\,,\, \nonumber \\
\eta_{y,k} &= (2\kappa_{xy} + \sqrt{\kappa_y-1})^{-1}/\mu_y
\,,\, \\
\eta_{u,k} &=(\sqrt{\kappa_x-1})^{-1}, \; \eta_{v,k}= (\sqrt{\kappa_y-1})^{-1} \,.\, \nonumber
\end{align}
Then for any $K > 0$, output of Algorithm \ref{algo:lpd} satisfies
\begin{align*}
&\epsilon (x^K,y^K)\\
&\leq \exp(- \frac{(K-1)}{(\kappa+1)}) \Big( \Big(\frac{1}{\eta_{x,0}}+\frac{L_x-\mu_x}{\eta_{u,0}}\Big) \|x^*-x_{0}\|^2 \;+ \nonumber \\
&\Big(\frac{1}{\eta_{y,0}}+\frac{L_y-\mu_y}{\eta_{v,0}}\Big) \|y^*-y_{0}\|^2\Big)\,.
\end{align*}
\end{theorem}
Note that the parameter choices in the above theorem are iteration ($k$) invariant.
The gradient complexity of Algorithm \ref{algo:lpd} is
\begin{equation}\label{eq:scsc_upperbound}
\Ord\Big(\Big(\sqrt{\frac{L_x}{\mu_x}-1} + \frac{\|A\|}{\sqrt{\mu_x\mu_y}}+ \sqrt{\frac{L_y}{\mu_y}-1}\Big) \log\Big(\frac{1}{\varepsilon}\Big)\Big),
\end{equation}
which is optimal and matches the lower-bound \cite{zhang2019lower} for Bi-SC-SC problem \eqref{eq:scsc_problem} up to logarithmic factors in the problem parameters. 
The lifting of the objective function allows the  PD method to be jointly applied to the smooth convex terms (as illustrated in Section~\ref{sec:32}) and to the bilinear minimax terms (as illustrated in Section~\ref{sec:31}), achieving this optimal rate \cite{zhang2019lower}. Comparisons to other  algorithms are given in Table~\ref{tab:complexity}.

We emphasize that LPD inherits the 
 computational and conceptual simplicity of the PD methods. %
The former leads to a single-loop algorithm, which is significantly simpler than other state-of-the-art complex multi-loop methods with sub-optimal guarantees  \cite{lin2020near,wang2020improved,xie2021dippa}.
The latter leads to a more transparent analysis, based on the simple analysis of the PD methods (Theorem \ref{thm:pd_rule}), 
which is based on an even simpler analysis of PPM (Lemma \ref{lem:ppm_rule}).

Note that we do not directly adapt the original guarantee of the PD method \cite{chambolle2016ergodic}. Our analysis has to be different since the naive application of the existing algorithm and analysis will depend on an effective strong convexity parameter (in $(x,v)$) of $\min(\mu_x, 1/(L_y - \mu_y))$, an effective strong concavity parameter (in $(y,u)$) of $\min(\mu_y, 1/(L_x - \mu_x))$, and a Lipschitz constant which is equal to the largest eigenvalue of the matrix effective coupling matrix
$[\mathbf{I}, A ; 0, \mathbf{I}]$.
This leads to a sub-optimal guarantee. Hence, we propose a different approach with a tighter analysis to achieve the optimal rates.

\begin{figure*}
	\begin{subfigure}{.25\textwidth}
	\centering
	\includegraphics[width=0.95\textwidth]{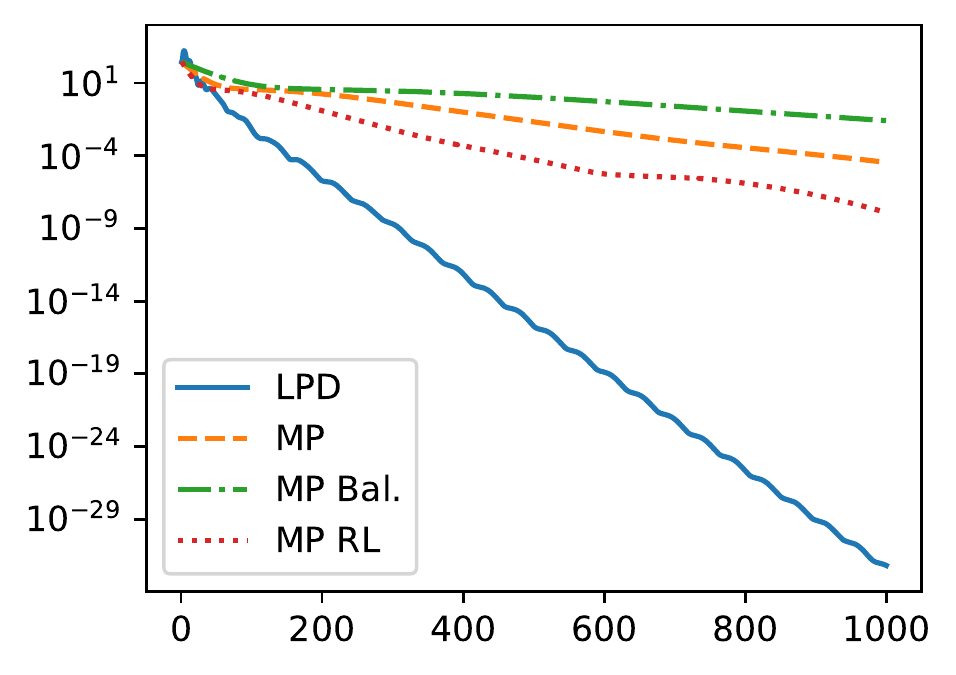}
	\put(-80,-5){\scalebox{.6}{{Number of Iterations {($K$)}}}}
	\put(-120, 20){\scalebox{.6}{\rotatebox{90}{{Primal-Dual gap}}}}%
	\caption{SC-SC Quadratic}	
	\label{fig:synthetic_scsc}
    \end{subfigure}%
    \hfill
	\begin{subfigure}{.25\textwidth}
	\centering
	\includegraphics[width=0.95\textwidth]{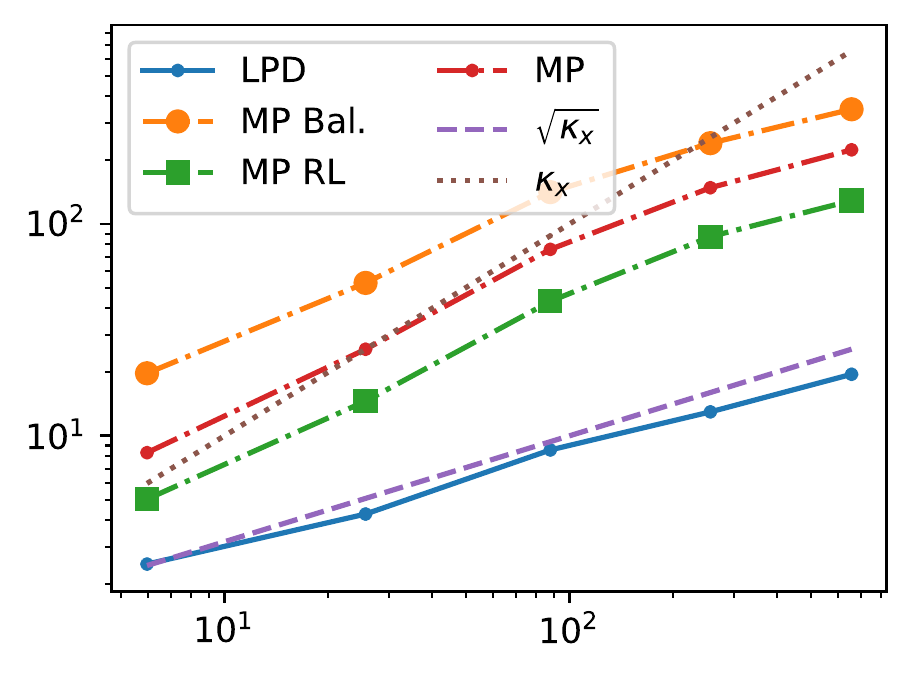}
	\put(-80,-5){\scalebox{.6}{{Condition number {($\kappa_x$)}}}}
	\put(-122, 25){\scalebox{.6}{\rotatebox{90}{{$K/\log(\Delta_0^2/\Delta_K^2)$}}}}%
	\vspace{-3pt}
	\caption{SC-SC Quadratic}
	\label{fig:synthetic_kappa}
    \end{subfigure}%
    \hfill
	\begin{subfigure}{.25\textwidth}
	\centering
	\includegraphics[width=0.95\textwidth]{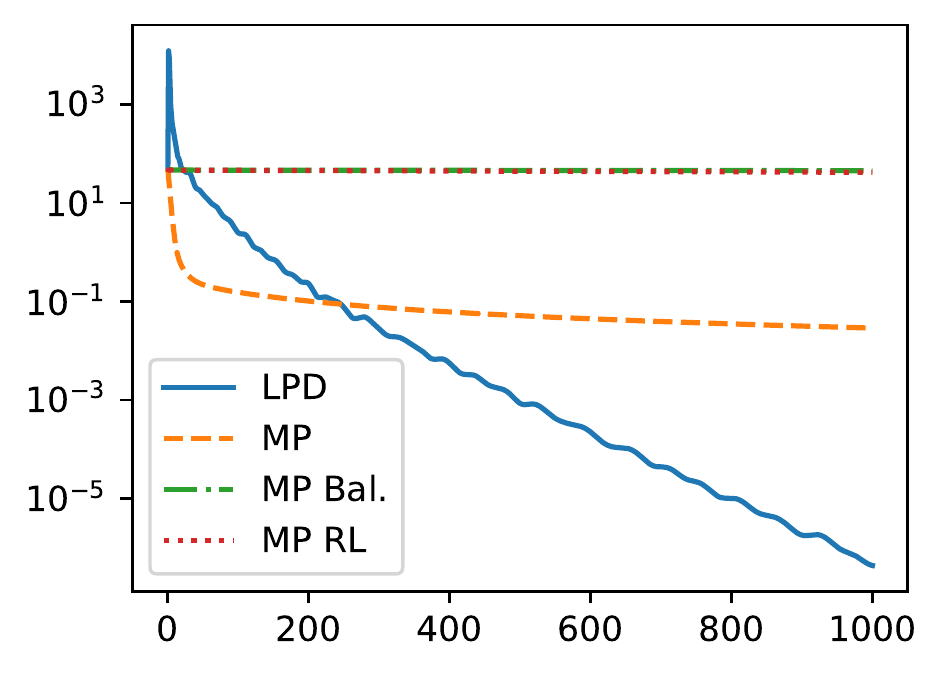}
	\put(-80,-5){\scalebox{.6}{{Number of Iterations {($K$)}}}}
	\put(-120, 20){\scalebox{.6}{\rotatebox{90}{{Primal-Dual gap}}}}%
	\caption{SC-SC Policy Evaluation}	
	\label{fig:rl_scsc}
    \end{subfigure}%
    \hfill
	\begin{subfigure}{.25\textwidth}
	\centering
	\includegraphics[width=0.95\textwidth]{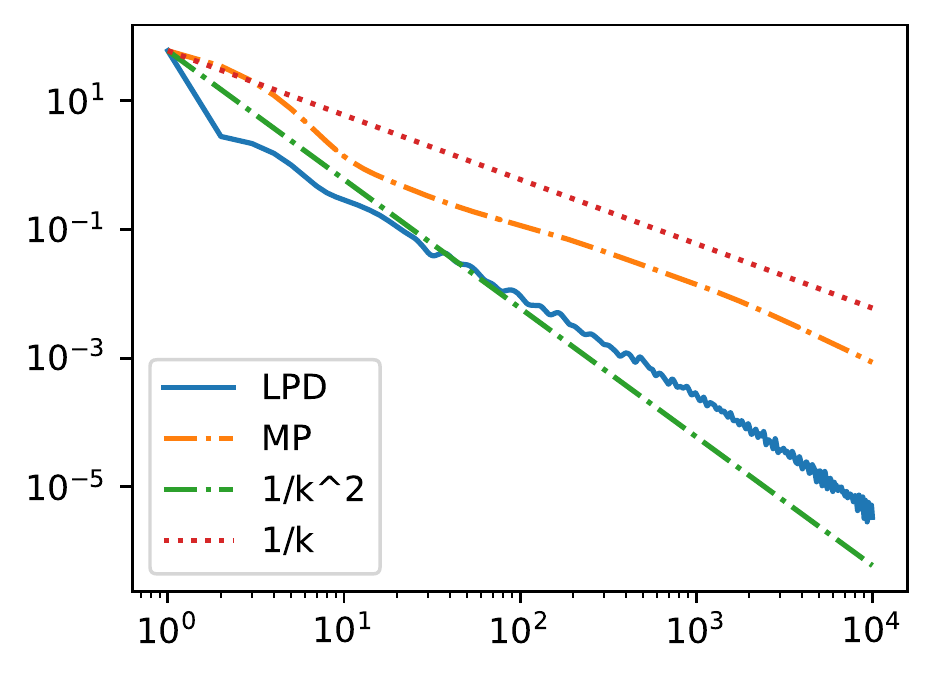}
	\put(-80,-5){\scalebox{.6}{{Number of Iterations {($K$)}}}}
	\put(-120, 20){\scalebox{.6}{\rotatebox{90}{{Primal-Dual gap}}}}%
	\caption{SC-C Policy Evaluation}	
	\label{fig:rl_csc}
    \end{subfigure}%
	\caption{LPD method (ours) achieves a faster linear convergence rate than competing algorithms in Strongly-Convex--Strongly-Concave synthetic quadratic minimax (a-b) and policy evaluation (c) problems. LPD method (ours) also achieves a faster $O(1/K^2)$ convergence rate than competing single-loop algorithm in Convex--Strongly-Concave policy evaluation problem (d).
	}    
\end{figure*}
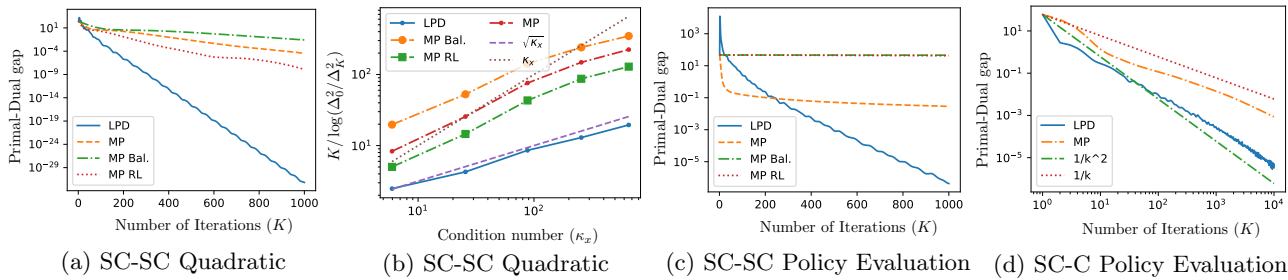
\textbf{Convex--Strongly-Concave Case.} 
Consider the Bilinearly-coupled Convex--Strongly-Concave (Bi-C-SC) case, where $f$ is merely convex, i.e.~$\mu_x = 0$. 
\begin{remark}[LPD + Smoothing \cite{nesterov2005smooth}] \label{rem:smooth_scsc_primal_dual_plus_smoothing}
Let $\phi(x,y)$ be the objective of a Bi-C-SC problem. Then we can apply LPD for Bi-SC-SC problems (Theorem \ref{thm:smooth_scsc_primal_dual_informal})
to the {\em smoothed} Bi-SC-SC objective $\phi(x,y) + \lambda \varepsilon \|x\|^2$ for some $\lambda > 0$,
and achieve an iteration complexity of
$\Ord( \sqrt{L_x/\varepsilon} + \|A\|/\sqrt{\mu_y \varepsilon} + \sqrt{L_y/\mu_y} )  \log\big(1/\varepsilon )$ for solving the original Bi-C-SC problem.
\end{remark}
The above result is optimal up to logarithmic factors.
The first term cannot be improved even for a pure minimization of convex $f$  \cite{nesterov2018lectures}. Due to a lower-bound of $\Omega(\|A\|/\sqrt{\mu_y \varepsilon})$ for the same problem when $f=0$ \cite{ouyang2021lower}, the second term cannot be improved. 
The third term cannot be improved even for a pure maximization of strongly-concave $h$ \cite{nesterov2018lectures}. 

However, smoothing might not be desirable in practice, because it requires bounded domains and fixing the final target error $\varepsilon$ in advance, and it is hard to tune $\lambda$ \cite{nesterov2005smooth}. We therefore design a direct algorithm by 
customizing the stepsizes of LPD. 
Let $D_\cX = \max_{x \in \cX \cap \domain(f)} \|x-x_{0}\|$ and $D_\cY = \max_{y \in \cY \cap \domain(h)} \|y-y_{0}\|$. Note that  the min variable solution $x^*$ may not be unique.

\begin{theorem}[Informal version of Corollary \ref{cor:smooth_csc_primal_dual}]
\label{thm:smooth_csc_primal_dual_informal}
Let
\begin{align}
1/{\eta_{x,k}} &= 1/{(k+1)\eta_x}\,,\,
1/{\eta_x} = 2L_x + {16\|A\|}/{\mu_y}
\,,\,
\nonumber \\
1/{\eta_{y,k}} &= 1/{(k+1)\eta_y} + {k\mu_y}/2\,,\,
1/{\eta_y} = 2(L_y-\mu_y)
\,,\,
\nonumber \\
\eta_{u,k} &= \eta_{v,k} = 2/{k}\,, \text{ for all } k \geq 0\,.
\end{align}
Then for any $K > 0$, output of Algorithm \ref{algo:lpd} satisfies \\
\noindent
$(a)$ if $D_\cX < \infty$ and $D_\cY < \infty$,
\begin{align}
&\max_{y \in \cY} \phi(\overline{x}_{K}, {y}) - \min_{x \in \cX} \phi({x}, \overline{y}_{K}) 
\leq \nonumber \\
&\frac{2 L_x D_\cX^2 }{K(K+1)} + \frac{16\|A\|^2 D_\cX^2 }{\mu_y K(K+1)}+
\frac{2(L_y-\mu_y)D_\cY^2}{K(K+1)}
\end{align}
where $(\overline{x}_K, \overline{y}_K) :=  {\sum_{k=1}^{K} \frac{2k}{K(K+1)} ( x_k, y_k)}$,
\\
\noindent
$(b)$ even if the feasible set is unbounded,
\begin{align}
\frac{\mu_y}{4}&\|y^*-y_{K}\|^2 
\leq \frac{2 L_x \|x^*-x_{0}\|^2}{K(K+1)}+ \nonumber \\
&\frac{16\|A\|^2 \|x^*-x_{0}\|^2}{\mu_y K(K+1)} +
\frac{2(L_y-\mu_y)\|y^*-y_{0}\|^2}{K(K+1)}
\end{align}
$(c)$ if $D_\cX < \infty$, $\phi_p(x)$$\,=\,$$\max_{y \in \cY} \phi(x, y)$, $\phi_d(x)$$\,=\,$$\min_{x \in \cX} \phi(x, y)$, and we do a warm restart on variable $y$ with $K_0 = \Omega_{\varepsilon}(1)$ initial additional iterations, then
\begin{align*}
&\phi_p(\overline{x}_{K}) - \phi_p(x^*)
\leq (L_x + \frac{10L_y\|A\|^2 }{\mu_y^2}) \frac{4 \|x^* - x_0\|^2}{K(K+1)} \,, \text{ and}
\nonumber \\
&\phi_d(y^*) - \phi_d(\overline{y}_{K})
\leq (L_x + \frac{8\|A\|^2 }{\mu_y}) \frac{4 D_\cX^2 }{K(K+1)}\,.
\end{align*}
\end{theorem}
This implies that, for Bi-C-SC problem, LPD has a gradient complexity of 
\begin{equation}\label{eq:csc_upperbound}
\Ord\Big(\Big(\sqrt{\frac{L_x}{\varepsilon}} + \frac{\|A\|}{\sqrt{\mu_y\varepsilon}}+ \sqrt{\frac{L_y-\mu_y}{\varepsilon}}\Big) \Big)\,.
\end{equation}

The LPD method achieves better complexities than previous single-loop algorithms \cite{mokhtari2020unified,nesterov2006solving}, PDHG-type algorithm \cite{zhao2019optimal}, direct multi-loop algorithm \cite{thekumparampil2019efficient}, and some smoothing-based multi-loop algorithms \cite{wang2020improved,lin2020near,xie2021dippa} (see Table \ref{tab:complexity}). 
Earlier single-loop methods such as \cite{chambolle2016ergodic,hamedani2021primal} achieve $\Ord(1/K^2)$ rate only under the restriction that $L_x$$\,=\,$$0$. This showcases the generality and simplicity of our LPD method, as it is the first single-loop algorithm which achieves $\Ord(1/K^2)$ for this problem. To the best of our knowledge, it is not known if better rates than in above theorem are achievable with a single-loop algorithm without using the smoothing technique, like in Lifted PD + Smoothing (Remark \ref{rem:smooth_scsc_primal_dual_plus_smoothing}). As discussed after Remark \ref{rem:smooth_scsc_primal_dual_plus_smoothing}, in practice, direct algorithms such as the one above are more desirable than smoothing-based algorithms.

\paragraph{Prox-friendly terms:} We point out that LPD can be extended to solve more general (possibly nonsmooth) minmax problems with the same guarantees:
\begin{align}\label{eq:scsc_general_problem}
\min_{x \in \cX} \max_{y \in \cY} \; F(x) + f(x) + \Ip{y}{Ax} - h(y) - H(y),
\end{align}
where $F$ and $H$ are convex and we have access to their proximal operators and $f,h$ satisfy our Assumption \ref{assume:ldp_main_assumptions}. We give the details of extension in Appendix \ref{sec:lpd_extension}.

\section{Experimental Results}
\label{sec:expts}
In this section we compare our LPD method with some competing single-loop non-smoothing-based direct algorithms when solving both synthetic and real-world problems. More details of the experiments are provided in Appendix \ref{sec:expt_details}.
First, we compare our LPD method with Mirror Prox (MP) \cite{mokhtari2020unified}, Balanced Mirror Prox (MP Bal.) (see Appendix \ref{sec:expt_details}), and Relative Lipschitzness-based Mirror Prox (MP RL) \cite{cohen2021relative} when solving Bi-SC-SC problems.  %
We only compared our (single-loop) algorithm with other single-loop algorithms, because multi-loop algorithms such as  \cite{wang2020improved} and \cite{xie2021dippa}  are typically challenging to implement and tune.  To the best of our knowledge, there are no publicly available implementations for these algorithms.

\textbf{Quadratic Problem:}
First, we consider synthetic quadratic problems of the form \eqref{eq:quadratic_problem}. We randomly generate the matrices $B$, $A$, $C$ in such a way that $\kappa_x = L_x/\mu_x = \kappa_y = L_y/\mu_y$ and $\kappa_{xy} = \|A\|/\sqrt{\mu_x \mu_y} := \sqrt{\kappa_x}$. In Figure \ref{fig:synthetic_scsc}, we plot the primal-dual gap against the number of iterations ($K$) of different algorithms when solving  such a problem with $\kappa_x = 256.0$. We see what LPD achieves a faster linear convergence than other methods. In Figure \ref{fig:synthetic_kappa}, we plot $K/\log(\Delta_0^2/\Delta_K^2)$ against $\kappa_x$ where $\Delta_K = \|x_K - x^*\|^2 + \|y_K - y^*\|^2$. We vary $\kappa_x$ from $5.96$ to $656.84$. %
As expected from theory, in this log-log scale plot, slope of the LPD curve is close to $1/2$ since $\Delta_K^2 \leq \Ord(\exp(-K/\sqrt{\kappa_x}))$ for LPD, and slope of other algorithms are close to one since $\Delta_K^2 \leq \Ord(\exp(-K/\kappa_x))$ for other algorithms.

\textbf{Policy Evaluation:}
Next, we consider policy evaluation problems of the form \eqref{eq:policy_eval_problem}. We consider the same MountainCar \cite{sutton2018reinforcement} reinforcement learning problem used in \cite{du2017stochastic}, and use the same copy of policy trace $\{(s_t,a_t,s_{t+1},r_t)\}_{t=1}^n$ used by \cite{du2017stochastic} to construct the MSPBE minimization problem. We create the feature vectors $\phi_t$, by applying PCA to the state vectors $s_t$ to whiten them. This reduces their dimension from $300$ to $200$. Finally setting $\rho=1.0$, results in a highly ill-conditioned Bi-SC-SC problem with $\kappa_x$$\,=\,$$1.0$, $\kappa_{xy}$$\,=\,$$24.35$, and $\kappa_y$$\,=\,$$19387.07$. In Figure \ref{fig:rl_scsc}, we plot the primal-dual gap against the number of iterations ($K$) of different algorithms when solving this problem. We observe that, our LPD method achieves much faster linear convergence than all other algorithm. Note that MP is better than LPD for small $K$, because in this regime the $\Ord(1/K)$ convergence rate of MP dominates its primal-dual gap.

Finally, we compare our LPD method with MP \cite{mokhtari2020unified}, when solving a Bi-SC-C problem.

\textbf{SC-C Policy Evaluation:} We consider the same minimum MSPBE estimation problem as above. However we directly use the $300$ dimensional state vectors $s_t$ as its feature vector $\phi_t$. This results in a Bi-SC-C problem. Note that Bi-SC-C objective is the negative of the objective of a Bi-C-SC problem, which means that we can solve it using LPD with stepsize choice given in Theorem \ref{thm:smooth_scsc_primal_dual_informal}. In Figure \ref{fig:rl_csc}, we plot the primal-dual gap against the number of iterations ($K$) of LPD and MP methods when solving this problem. As theory predicts, we observe that the LPD method achieves a much faster $O(1/K^2)$ convergence rate than $O(1/K)$ convergence rate of MP.

\section{Conclusion}
We studied Bi-SC-SC problem and provided an optimal single-loop algorithm: the Lifted Primal-Dual (LPD) method to solve it. The LPD method is designed using simple building blocks of the Primal-Dual method and \textit{lifting}, leading to its generalizability, simplicity, and transparent analysis. Further, we also provide two related algorithms---one optimal (up to logarithmic factors) and another single-loop---to solve Bi-C-SC problem.

\section*{Acknowledgement} 

This work is supported by Google faculty research award and NSF grants CNS-2002664, IIS-1929955, DMS-2134012, CCF-2019844 as a part of NSF Institute for Foundations of Machine Learning (IFML), and CNS-2112471 as a part of NSF AI Institute for Future Edge Networks and Distributed Intelligence (AI-EDGE). This work was done prior to the first author joining Amazon, and it does not relate to his current position there.

\printbibliography

\onecolumn

\appendix

\section{Definitions and Standard results}
\label{sec:prelim_details}

\subsection{Convexity and Smoothness}

\begin{definition}
We say that a function is $\mu$-strongly convex if
\begin{align}
&f(\alpha x_1 + (1-\alpha)x_2) \leq \alpha f(x_1) + (1-\alpha) f(x_2) - \frac{\mu}2 \alpha (1-\alpha) \|x_1-x_2\|^2 \,,\;\; \text{ for any } \alpha \in [0,1]\,, \nonumber \\
&f(x_2) \geq f(x_1) + \Ip{f'(x_1)}{x_2 - x_2} + \frac{\mu}2 \|x_1-x_2\|^2 \,, \text{ or equivalently } \nonumber \\
&\Ip{f'(x_1)(x_1) - f'(x_1)(x_2)}{x_1 - x_2} \geq \mu \|x_1 - x_2\|^2 \nonumber
\end{align}
for all $x_1$ and $x_2$, where at any point $x$,  $f'(x) \in \partial f(x)$ is some sub-gradient of the function in its (Frechet) sub-differential $\partial f(x)$ at that point. Further we say that a function (merely) convex if it is $0$-strongly convex.
\end{definition}

For a differentiable function $f$, its gradient at any point $x$ is denoted by $\nabla f(x)$.
\begin{definition}
We say that a function is $L$-smooth if it is differentiable and
\begin{align}
f(x_2) \leq f(x_1) + \Ip{\nabla f(x_1)}{x_2 - x_2} + \frac{L}2 \|x_2 - x_1\|^2 \,, \text{ or equivalently } \Ip{\nabla f(x_1) - \nabla f(x_2)}{x_1 - x_2} \leq L \|x_1 - x_2\|^2 \nonumber
\end{align}
for all $x_1$ and $x_2$, where where at any point $x$, $\nabla f(x)$ is gradient of the function at that point $x$. 
\end{definition}

\subsection{Fenchel/Convex Conjugate and Duality}

\begin{definition}
Let $f$ be a convex function. Then its Fenchel/convex conjugate $f^*$ is defined as $f^*(u) := \max_{x} \Ip{u}{x} - f(x)$
\end{definition}

\begin{lemma}[\cite{nesterov2018lectures,kakade2009duality}] \label{lem:fenchel_props}
Fenchel/convex conjugate satisfy the following properties.
\begin{enumerate}
\item[(a)] If $f$ is an $L$-smooth and convex function, then $f^*$ $1/L$-strongly convex.
\item[(b)] If $f$ is an $L$-smooth and convex function, then $\nabla f(x) = \arg\min_{x} \Ip{x}{u} - f^*(u)$.
\item[(c)] If $f$ is a convex function, $(f^*)^*$ is $f$
\item[(d)] If $f$ is an $L$-smooth and convex function and $u = \nabla f(x)$ then $x = \arg\min_{u} \Ip{u}{x} - f^*(x) \in \partial (f^*)(u)$.
\end{enumerate}
\end{lemma}

\subsection{Proximal Operator}
\begin{definition}
For a convex function, $F$, its proximal operator $\prox_{\eta F}(x)$ (parameterized by some $\eta > 0$) is defined as
\begin{align}
\prox_{\eta F}(x) = \arg\min_{\tx} F(x) + \frac1{2\eta} \|\tx - x\|^2
\end{align}
\end{definition}

\subsection{Bregman Divergence, and Relative Lipschitzness and Relative Convexity}

\begin{definition}
Let $r$ be a strongly convex function. Then Bregman divergence $V^{r}_x(\tx)$ w.r.t.~to the distance generating function (d.g.f.) $r$ is defined as the
\begin{align}
V^{r}_{x}(\tx) = r(\tx) - r(x) - \Ip{r'(x)}{\tx - x}
\end{align}
where $r'(x) \in \partial r(x)$ is a sub-gradient of $r$ at $x$.
\end{definition}

\begin{lemma} \label{lem:breg_props}
Let $r$ be a $\sigma$-strongly convex function. Then Bregman divergence $V^{r}_x(\tx)$ w.r.t.~to the d.g.f.~$r$ satisfies $V^{r}_{x}(\tx) \geq ({\sigma}/2) \|\tx - x\|^2$.
\end{lemma}

\begin{lemma}\label{lem:dual_bregman_div}
If $f$ $L$-smooth convex function, $u = \nabla f(x)$ and $u_0 = \nabla f(x_0)$, then $\frac1{2L} \|u-u_0\|^2 \leq V^{f^*}_{u_0}(u) = V^{f}_x(x_0) \leq L/2 \|x-x_0\|^2$
\end{lemma}
\begin{proof}
By Lemma \ref{lem:fenchel_props}, $f^*$ is $1/L$-strongly convex. Then using Lemma \ref{lem:breg_props} and Lemma \ref{lem:fenchel_props} we can show that.
\begin{align}
\frac1{2L} \|u-u_0\|^2 \leq V^{f^*}_{u_0}(u) &= f^*(u) - f^*(u_0) - \Ip{{f^*}'(u_0)}{u-u_0} \\
&= (\Ip{u}{x} - f(x)) - (\Ip{u_0}{x_0} - f(x_0)) - \Ip{x_0}{u-u_0} \\
&= f(x_0) - f(x) - \Ip{u}{x_0-x} \\
&= f(x_0) - f(x) - \Ip{\nabla f(x)}{x_0-x} \\
&= V^f_{x_0}(x) \\
&\leq \frac{L}2 \|x-x_0\|^2
\end{align}
\end{proof}

\begin{definition}
We say that a function is relatively $\mu$-strongly convex w.r.t.~to a Bregman divergence $V^r$ (generated by a strongly convex d.g.f.~$r$) if
\begin{align}
&f(\alpha x_1 + (1-\alpha)x_2) \leq \alpha f(x_1) + (1-\alpha) f(x_2) - {\mu} \alpha (1-\alpha) V^r_{x_1}(x_2) \,,\;\; \text{ for any } \alpha \in [0,1]\,, \nonumber \\
&f(x_2) \geq f(x_1) + \Ip{f'(x_1)}{x_2 - x_2} + {\mu} V^r_{x_1}(x_2) \,, \text{ or equivalently } \nonumber \\
&\Ip{f'(x_1)(x_1) - f'(x_1)(x_2)}{x_1 - x_2} \geq 2 \mu V^r_{x_1}(x_2) \nonumber
\end{align}
for all $x_1$ and $x_2$, where at any point $x$, $f'(x) \in \partial f(x)$ is some sub-gradient of the function in its (Frechet) sub-differential $\partial f(x)$ at that point. Further we say that a function (merely) relatively convex  w.r.t.~to the Bregman divergence $V^r$ if it is relatively $0$-strongly convex w.r.t.~$V^r$ .
\end{definition}

\begin{definition}
We say that a convex function, $f$ is relatively smooth w.r.t.~to a Bregman divergence $V^r$ (generated by a strongly convex d.g.f.~$r$)
\begin{align}
f(x_2) \leq f(x_1) + \Ip{\nabla f(x_1)}{x_2 - x_2} + L V^r_{x_1}(x_2) \,, \text{ or equivalently } \Ip{\nabla f(x_1) - \nabla f(x_2)}{x_1 - x_2} \leq 2L V^r_{x_1}(x_2) \nonumber
\end{align}
\end{definition}

\begin{definition}
For a convex function, $F$, its relative proximal operator $\prox^{r}_{\eta F}(x)$ (parameterized by some $\eta > 0$) w.r.t.~to a Bregman divergence $V^r$ (generated by a strongly convex d.g.f.~$r$) is defined as
\begin{align}
\prox^{r}_{\eta F}(x) = \arg\min_{\tx} F(x) + \frac1{\eta} V^{r}_{x}(\tx)
\end{align}
\end{definition}

\subsection{Minimax Problems}

\begin{lemma}\label{lem:minimax_pseudo_gap}
Let $\phi(x,y)$ be convex-concave objective. Then $\phi(\tx, y^*) - \phi(x^*, \ty) \geq 0$ for all $(\tx,\ty) \in \cX \times \cY$, if $(x^*, y^*) \in \arg\min_{x \in \cX, y \in \cY} \phi(x,y)$.
\end{lemma}
\begin{proof}
Notice that the LHS above is positive since
\begin{align}
\phi(\tx, y^*) - \phi(x^*, \ty) &= (\phi(\tx, y^*) - \phi(x^*,y^*)) + (\phi(x^*,y^*) - \phi(x^*, \ty)) \nonumber \\
&= (\phi(\tx, y^*) - \min_x \phi(x,y^*)) + (\max_y \phi(x^*, y) - \phi(x^*,\ty)) \nonumber \\
&\geq 0
\end{align}
\end{proof}
\section{Supporting Results}

\subsection{Proximal Point method: Proof of Lemma \ref{lem:ppm_rule}}
\label{sec:ppm_proof}

\begin{proof}
Since $F$ ($H$) is $\mu$-relatively strong convexity w.r.t.~$r$ ($s$) and $(x_{k+1}, y_{k+1})$ satisfies PPM rule \eqref{eq:ppm_rule}, we can use mirror descent lemma \ref{lem:md_lemma} to get
\begin{align}
F(x_{k+1}) - F(x) + \Ip{A^\top y_{k+1}}{x_{k+1} - x} &\leq
\frac1{\eta_x} V^{r}_{x_{k}}(x) - (\frac1{\eta_x} + {\mu_x})V^{r}_{x_{k+1}}(x) - \frac1{\eta_x} V^{r}_{x_{k}}(x_{k+1}) \nonumber \\
H(y_{k+1}) - H(y) - \Ip{A x_{k+1}}{y_{k+1} - y} &\leq
\frac1{\eta_y} V^{s}_{y_{k}}(y) - (\frac1{\eta_y} + {\mu_y})V^{s}_{y_{k+1}}(y) - \frac1{\eta_y} V^{s}_{y_{k}}(y_{k+1})
\end{align}
Separately, using convexity and concavity of $\Ip{y}{Ax}$ w.r.t.~$x$ and $y$, we get
\begin{align}
\phi(x_{k+1}, y) - \phi(x, y_{k+1})
&\leq F(x_{k+1}) - F(x) + \Ip{A^\top y_{k+1}}{x_{k+1} - x} + \Ip{A x_{k+1}}{y_{k+1} - y} + H(y_{k+1}) - H(x)
\end{align}
Summing three equations and setting $(x,y) = (x^*, y^*)$ we get
\begin{align}\label{eq:ppm_eq1}
\phi(x_{k+1}, y^*) - \phi(x^*, y_{k+1})
&\leq \frac1{\eta_x} V^{r}_{x_{k}}(x^*) - (\frac1{\eta_x} + {\mu_x})V^{r}_{x_{k+1}}(x^*) +  \frac1{\eta_y} V^{s}_{y_{k}}(y^*) - (\frac1{\eta_y} + {\mu_y})V^{s}_{y_{k+1}}(y^*)
\end{align}
Notice that the LHS above is positive by Lemma \ref{lem:minimax_pseudo_gap}, that is
\begin{align}
\phi(x_{k+1}, y^*) - \phi(x^*, y_{k+1}) 
&\geq 0
\end{align}
Let us define $\gamma := 1 + \kappa^{-1}$, where we define also $\kappa := 1/\min(\eta_x \mu_x, \eta_y \mu_y)$.
Now multiplying both the sides of \eqref{eq:ppm_eq1} with $\gamma^{k}$, and using $\gamma \leq 1+\min(\eta_x \mu_x, \eta_y \mu_y)$ we get 
\begin{align}
0 &\leq \gamma^{k} (\phi(x_{k+1}, y^*) - \phi(x^*, y_{k+1})) \nonumber \\
&\leq \frac{\gamma^k}{\eta_x} V^{r}_{x_{k}}(x^*) - \frac{\gamma^{k+1}}{\eta_x} V^{r}_{x_{k+1}}(x^*) +  \frac{\gamma^k}{\eta_y} V^{s}_{y_{k}}(y^*) - \frac{\gamma^{k+1}}{\eta_y} V^{s}_{y_{k+1}}(y^*)\,.
\end{align}
Now summing the above equation from $k=0$ to $k=K-1$, we get that
\begin{align}
\frac{\gamma^{K}}{\eta_x} V^{r}_{x_{K}}(x^*) +  \frac{\gamma^{K}}{\eta_y} V^{s}_{y_{K}}(y^*)
&\leq \frac{1}{\eta_x} V^{r}_{x_{0}}(x^*) + \frac{1}{\eta_y} V^{s}_{y_{0}}(y^*) \,.
\end{align}
Finally, dividing both sides using $2\gamma^K$ and using the $1$-strongly convexity of $r$ and $s$ and Lemma~\ref{lem:breg_props}(a) we get
\begin{align}
\frac{1}{\eta_x} \|x^* - x_k\|^2 +  \frac{1}{\eta_y} \|y^* - y_k\|^2
&\leq \frac{2\gamma^{-K}}{\eta_x} V^{r}_{x_{0}}(x^*) + \frac{2\gamma^{-K}}{\eta_y} V^{s}_{y_{0}}(y^*) 
\end{align}
Finally we get the desired result using the fact that $\gamma^{-1} = 1/(1 + \kappa^{-1}) = 1 - 1/(1 + \kappa) \leq \exp(1/\kappa+1)$.
\end{proof}

\subsection{Primal Dual Method: Proof of Theorem \ref{thm:pd_rule}}
\label{sec:pd_proof}
Since the PD method is an approximation of PPM, former's analysis closely follows that of the latter (proof of Lemma \ref{lem:ppm_rule}).
\begin{proof}
Since $F$ ($H$) is $\mu$-relatively strong convexity w.r.t.~$r$ ($s$) and $(x_{k+1}, y_{k+1})$ satisfies PPM rule \eqref{eq:ppm_rule}, we can use mirror descent lemma \ref{lem:md_lemma} to get
\begin{align}
F(x_{k+1}) - F(x) + \Ip{A^\top \ty_{k+1}}{x_{k+1} - x} &\leq
\frac1{\eta_x} V^{r}_{x_{k}}(x) - (\frac1{\eta_x} + {\mu_x})V^{r}_{x_{k+1}}(x) - \frac1{\eta_x} V^{r}_{x_{k}}(x_{k+1}) \nonumber \\
H(y_{k+1}) - H(y) - \Ip{A x_{k+1}}{y_{k+1} - y} &\leq
\frac1{\eta_y} V^{s}_{y_{k}}(y) - (\frac1{\eta_y} + {\mu_y})V^{s}_{y_{k+1}}(y) - \frac1{\eta_y} V^{s}_{y_{k}}(y_{k+1})
\end{align}
Separately, using convexity and concavity of $\Ip{y}{Ax}$ w.r.t.~$x$ and $y$, we get
\begin{align}
\phi(x_{k+1}, y) - \phi(x, y_{k+1})
&\leq F(x_{k+1}) - F(x) + \Ip{A^\top y_{k+1}}{x_{k+1} - x} + \Ip{A x_{k+1}}{y_{k+1} - y} + H(y_{k+1}) - H(x)
\end{align}
Summing above three equations and setting $(x,y) = (x^*, y^*)$ we get
\begin{align}
\phi(x_{k+1}, y^*) - \phi(x^*, y_{k+1})
&\leq \frac1{\eta_x} V^{r}_{x_{k}}(x^*) - (\frac1{\eta_x} + {\mu_x})V^{r}_{x_{k+1}}(x^*) +  \frac1{\eta_y} V^{s}_{y_{k}}(y^*) - (\frac1{\eta_y} + {\mu_y})V^{s}_{y_{k+1}}(y^*) \;+ \nonumber \\
&\;\;\;\;\; \Ip{A^\top (y_{k+1} - \ty_{k+1})}{x_{k+1} - x^*} - \frac{1}{\eta_{x}} V^{r}_{x_{k}}(x_{k+1}) - \frac{1}{\eta_{v}} V^{s}_{y_{k}}(y_{k+1})
\label{eq:scsc_pd_pf_eq1}
\end{align}

We can further expand out the last four term in the above inequality as follows. Using $\ty_{k+1} = y_{k} + \theta (y_{k} - y_{k-1})$ (equation \eqref{eq:pd_rule}) and Cauchy-Schwarz inequality we get
\begin{align}
\Ip{A^\top(y_{k+1} - \ty_{k+1})}{x_{k+1} - x^*} &= \theta_k \Ip{y_{k-1} - y_{k}}{A(x_{k} - x^*)} - \Ip{y_{k} - y_{k+1}}{A(x_{k+1} - x^*)} \;+ \nonumber \\
&\;\;\;\;\;\; \theta \Ip{y_{k-1} - y_{k}}{A(x_{k+1} - x_{k})} \nonumber \\
&\leq \theta \Ip{y_{k-1} - y_{k}}{A(x_{k} - x^*)} - \Ip{y_{k} - y_{k+1}}{A(x_{k+1} - x^*)} \;+ \nonumber \\
&\;\;\;\;\;\; \frac{\theta \|A\| \alpha_{y}}2 \|y_{k-1} - y_{k}\|^2 + \frac{\theta \|A\| }{2\alpha_{y}} \|x_{k+1} - x_{k}\|^2
\label{eq:scsc_pd_pf_eq2a}
\end{align}
for some $\alpha_y = \sqrt{{\mu_x}/{\mu_y}}$.
Using Lemma \ref{lem:breg_props}(a) and $1$-strong convexity of $\uf^*$ we get that
\begin{align}
- \frac{1}{\eta_{x}} V^{r}_{x_{k}}(x_{k+1})\leq -\frac{1}{2\eta_{x}} \|x_{k+1} - x_{k}\|^2 \label{eq:scsc_pd_pf_eq3a} \\
- \frac{1}{\eta_{y}} V^{s}_{y_{k}}(y_{k+1})\leq -\frac{1}{2\eta_{y}} \|y_{k+1} - y_{k}\|^2 \label{eq:scsc_pd_pf_eq3b}
\end{align}
Summing equations \eqref{eq:scsc_pd_pf_eq2a}, \eqref{eq:scsc_pd_pf_eq3a} and \eqref{eq:scsc_pd_pf_eq3b}, and using $\frac{\alpha_x \|A\|}{2} = \frac{\|A\|}2 \sqrt{\frac{\mu_x}{\mu_y}} \leq \|A\| \sqrt{\frac{\mu_x}{\mu_y}} = \frac1{2\eta_{y}}$ and $\theta \frac{\|A\|}{2\alpha_y} = \theta \frac{\|A\|}2 \sqrt{\frac{\mu_y}{\mu_x}} \leq \|A\| \sqrt{\frac{\mu_y}{\mu_x}} = \frac1{2\eta_{x}}$ we get
\begin{align}
&\Ip{A^\top (y_{k+1} - \ty_{k+1})}{x_{k+1} - x^*} - \frac{1}{\eta_{x}} V^{r}_{x_{k}}(x_{k+1}) - \frac{1}{\eta_{v}} V^{s}_{y_{k}}(y_{k+1}) \nonumber \\
\;\leq\; &\theta \Ip{y_{k-1} - y_{k}}{A(x_{k} - x^*)} - \Ip{y_{k} - y_{k+1}}{A(x_{k+1} - x^*)} \;+ \nonumber \\
&- (\frac{1}{2\eta_{x}} - \theta \frac{\|A\| }{2\alpha_{y}}) \|x_{k+1} - x_{k}\|^2 \;+ 
\theta \frac{\|A\| \alpha_{y}}2 \|y_{k} - y_{k-1}\|^2 - \frac{1}{2\eta_{y}} \|y_{k+1} - y_{k}\|^2 \nonumber \\
\;\leq\; &\theta \Ip{y_{k-1} - y_{k}}{A(x_{k} - x^*)} - \Ip{y_{k} - y_{k+1}}{A(x_{k+1} - x^*)} \;+ \nonumber \\
&\theta \frac{\|A\| \alpha_{y}}2 \|y_{k} - y_{k-1}\|^2 - \frac{\|A\| \alpha_{y}}2 \|y_{k+1} - y_{k}\|^2 \label{eq:scsc_pd_pf_eq4}
\end{align}
Notice that the LHS of \eqref{eq:scsc_pd_pf_eq1} is positive by Lemma \ref{lem:minimax_pseudo_gap}, that is $\phi(x_{k+1}, y^*) - \phi(x^*, y_{k+1}) \geq 0$.
Summing equations \eqref{eq:scsc_pd_pf_eq1} and \eqref{eq:scsc_pd_pf_eq4},and using the above fact we get
\begin{align}
0 \;\leq\; &\phi(x_{k+1}, y^*) - \phi(x^*, y_{k+1}) \nonumber \\
\;\leq\; &\frac{1}{\eta_{x}}V^{r}_{x_k}(x^*) - (\frac1{\eta_{x}} + \mu_x) V^{r}_{x_{k+1}}(x^*) +
\frac{1}{\eta_{y}}V^{s}_{y_{k}}(y) - (\frac1{\eta_{y}} + \mu_y) V^{s}_{y_{k+1}}(y) \;+ \nonumber \\
&\theta \Ip{y_{k-1} - y_{k}}{A(x_{k} - x^*)} - \Ip{y_{k} - y_{k+1}}{A(x_{k+1} - x^*)} \;+ \nonumber \\
&\theta \frac{\|A\| \alpha_{y}}2 \|y_{k} - y_{k-1}\|^2 - \frac{\|A\| \alpha_{y}}2 \|y_{k+1} - y_{k}\|^2 \label{eq:scsc_pd_pf_eq5}
\end{align}
Let us define $\gamma := 1 + \kappa^{-1}$, where we define also $\kappa := 1/\min(\eta_x \mu_x, \eta_y \mu_y) = 2\|A\|/\sqrt{\mu_x \mu_y}$.
Now multiplying both the sides of \eqref{eq:scsc_pd_pf_eq5} with $\gamma^{k}$, and using $\gamma \leq 1+\min(\eta_x \mu_x, \eta_y \mu_y)$ and $\theta = 1/\gamma$ we get 
\begin{align}
0 \leq &\frac{\gamma^k}{\eta_x} V^{r}_{x_{k}}(x^*) - \frac{\gamma^{k+1}}{\eta_x} V^{r}_{x_{k+1}}(x^*) +  \frac{\gamma^k}{\eta_y} V^{s}_{y_{k}}(y^*) - \frac{\gamma^{k+1}}{\eta_y} V^{s}_{y_{k+1}}(y^*) \nonumber \\
&\gamma^{k-1} \Ip{y_{k-1} - y_{k}}{A(x_{k} - x^*)} - \gamma^{k} \Ip{y_{k} - y_{k+1}}{A(x_{k+1} - x^*)} \;+ \nonumber \\
&\gamma^{k-1} \frac{\|A\| \alpha_{y}}2 \|y_{k} - y_{k-1}\|^2 -  \gamma^{k} \frac{\|A\| \alpha_{y}}2 \|y_{k+1} - y_{k}\|^2
\end{align}
Now summing the above equation from $k=0$ to $k=K-1$ and using $y_{-1}=y_0$, we get that
\begin{align}
\frac{\gamma^{K}}{\eta_x} V^{r}_{x_{K}}(x^*) +  \frac{\gamma^{K}}{\eta_y} V^{s}_{y_{K}}(y^*)
&\leq \frac{1}{\eta_x} V^{r}_{x_{0}}(x^*) + \frac{1}{\eta_y} V^{s}_{y_{0}}(y^*) \;+ \nonumber \\
&\;\;\;\; \gamma^{K-1} \Ip{y_{K-1} - y_{K}}{A(x_{K} - x^*)} +  \gamma^{K-1} \frac{\|A\| \alpha_{y}}2 \|y_{K} - y_{K-1}\|^2 
\end{align}
Using Cauchy-Schwarz inequality, $\theta \frac{\|A\|}{2\alpha_y} = \theta \frac{\|A\|}2 \sqrt{\frac{\mu_y}{\mu_x}} \leq \|A\| \sqrt{\frac{\mu_y}{\mu_x}} = \frac1{2\eta_{x}}$, Lemma \ref{lem:breg_props}(a) we can show that
\begin{align}
\gamma^{K-1} &\Ip{y_{K-1} - y_{K}}{A(x_{K} - x^*)} -  \gamma^{K-1} \frac{\|A\| \alpha_{y}}2 \|y_{k+1} - y_{k}\|^2 \nonumber \\
&\leq \gamma^{K-1} \frac{\|A\| }{2\alpha_{y}} \|x_{K} - x^*\|^2 + \gamma^{K-1} \frac{\|A\| \alpha_{y}}2 \|y_{K} - y_{K-1}\|^2 -  \gamma^{K-1} \frac{\|A\| \alpha_{y}}2 \|y_{K} - y_{K-1}\|^2 \nonumber \\
&\leq \gamma^{K-1} \frac{\|A\| }{2\alpha_{y}} \|x_{K} - x^*\|^2 \nonumber \\
&\leq \gamma^{K} \frac{1}{4\eta_{y}} \|x_{K} - x^*\|^2 \nonumber \\
&\leq \gamma^{K} \frac{1}{2\eta_{y}} V^{r}{x_{K}} (x^*) \label{eq:scsc_pd_pf_eq6}
\end{align}
Summing equations \eqref{eq:scsc_pd_pf_eq6} and \eqref{eq:scsc_pd_pf_eq7}, and dividing both sides of the resulting equation using $2\gamma^K$ and using the $1$-strongly convexity of $r$ and $s$ and Lemma~\ref{lem:breg_props}(a) we get
\begin{align}
\frac{1}{2\eta_x} \|x^* - x_k\|^2 +  \frac{1}{\eta_y} \|y^* - y_k\|^2
&\leq \frac{2\gamma^{-K}}{\eta_x} V^{r}_{x_{0}}(x^*) + \frac{2\gamma^{-K}}{\eta_y} V^{r}_{y_{0}}(y^*) \label{eq:scsc_pd_pf_eq7}
\end{align}
Finally we get the desired result by using the choice $\gamma = 1+\min(\eta_x \mu_x, \eta_y \mu_y) = 1 + \kappa^{-1}$, which implies that $\gamma^{-1} = 1/(1 + \kappa^{-1}) = 1 - 1/(1 + \kappa) \leq \exp(1/\kappa+1)$.
\end{proof}

\subsection{Proof of Lemma \ref{lem:sc_minus_quad}}
\label{sec:sc_minux_quad_pf}
\begin{lemma}\label{lem:sc_minus_quad}
If $f$ is $\mu$-strongly convex and $L$-smooth, then $\uf = f- \mu \|\cdot\|^2/2$ is convex and $(L-\mu)$-smooth.
\end{lemma}
\begin{proof}
It can be easily proved by noticing that
\begin{align}
\Ip{\nabla \uf(x) - \nabla \uf(\tx)}{x-\tx} &= \Ip{\nabla f(x) - \nabla f(\tx)}{x-\tx} + \Ip{-\mu x + \mu \tx}{x-\tx} \nonumber \\
&\leq (L-\mu) \|x-\tx\|^2
\end{align}
Similarly we can also easily show that $\Ip{\nabla \uf(x) - \nabla \uf(\tx)}{x-\tx} \geq 0$
\end{proof}

\subsection{Proof of Corollary \ref{cor:sc_agd}}
\label{sec:sc_agd_pf}

We omit the proof of Corollary \ref{cor:sc_agd} since it is very similar to that of Theorem \ref{thm:pd_rule}. Only additional step is to upper-bound $V^{\uf^*}_{u_0}(u^*)$ by $(L_x - \mu_x) \|x_0 - x^*\|^2/2$ using $u^* = \nabla \uf(x^*)$ and $u_0 = \nabla \uf (x_0)$ and Lemma \ref{lem:dual_bregman_div}.

\subsection{Proof of Lemma \ref{lem:sc_agd_update}}
\label{sec:sc_agd_update_pf}

\begin{proof}
We want to prove that $x_k$ iterates of \eqref{eq:pd_rule} (repeated below)
\begin{equation}
\left\{
\begin{aligned}
\tu_{k+1} &= u_{k} + \theta (u_{k} - u_{k-1}) \\
x_{k+1} &= \arg\min_{x} \Ip{\tilde{u}_{k+1}}{x} +\frac{\mu}{2} \|x\|^2+\frac{1}{2\eta_x}\|x-x_k\|^2  \\
u_{k+1} &= \arg\min_{u} -\Ip{x_{k+1}}{u} + \uf^*(u) + \frac1{\eta_u} V^{\uf^*}_{u_{k}}(u)
\end{aligned}
\right.
\end{equation}
and \eqref{eq:sc_equiv_agd_rule} (repeated below)
\begin{equation}
\left\{
\begin{aligned}
\tnabla_{k+1} &= \nabla \uf(\ux_{k}) + \theta (\nabla \uf(\ux_{k}) - \nabla \uf(\ux_{k-1}))  \\
x_{k+1} &= (x_k - \eta_x \tnabla_{k+1})/(1+\eta_x \mu) \\
\ux_{k+1} &= ({\ux_{k} + \eta_u x_{k+1}})/({1+\eta_u}) 
\end{aligned}
\right.
\end{equation}
are equivalent under the given condition. For this we will prove a stronger condition which additionally states that $u_{k} = \nabla \uf(\ux_k)$ for all $k = -1,0,1,\ldots$.

We prove this by induction. Let us initialize both the updates using $x_0$. For the base case it is easy to see that when $u_{-1} = u_0 = \nabla f(\ux_{-1}) = \nabla f(\ux_{0}) = \nabla f(x_0)$.

Let $u_{\widetilde{k}-1} = \nabla f(\ux_{\widetilde{k}-1})$ and $u_{\widetilde{k}} = \nabla f(\ux_{\widetilde{k}})$, and $x_{\widetilde{k}}$ iterates of the both the rules match for $\widetilde{k} = 0,1,\ldots,k$. Then clearly, $\tu_{k+1} = \tnabla_{k+1}$. This implies that $x_{k+1}$ iterates are the same for both the rules.

Next we will prove that $u_{k+1} = \nabla f(\ux_{k+1})$. Note that $u_{k+1} = \arg\min_{u} -\Ip{x_{k+1}}{u} + \uf^*(u) + \frac1{\eta_u} V^{\uf^*}_{u_{k}}(u)$. However, $V^{\uf^*}_{u_{k}}(u) = \uf^*(u) - \uf^*(u_k) - \Ip{(\uf^{*})'(u_k)}{u-u_k}$ is not defined unless we fix a sub-gradient $(\uf^{*})'(u_k) \in \partial \uf^*(u_k)$ at $u_k$. For making the rules equivalent we set $(\uf^{*})'(u_k) = \ux_k$. Note that $\ux_k \in \partial \uf^*(u_k)$ since $u_k = \nabla \uf(\ux_k)$ (Lemma \ref{lem:fenchel_props}(d)). Then $u_{k+1} = \nabla f(\ux_{k+1})$, since
\begin{align}
u_{k+1} &= \arg\min_{u} -\Ip{x_{k+1}}{u} + \uf^*(u) + \frac1{\eta_u} V^{\uf^*}_{u_{k}}(u) \nonumber \\
&= \arg\min_{u} -\Ip{\eta_u x_{k+1} + (\uf^{*})'(u_k)}{u} + (1+\eta_u) \uf^*(u) \nonumber \\
&= \arg\min_{u} -\Ip{\eta_u x_{k+1} + \ux_k}{u} + (1+\eta_u) \uf^*(u)
\end{align}
and by Lemma \ref{lem:fenchel_props}(b) $u_{k+1} = \nabla f(\ux_{k+1})$ is a valid and only solution (because of strong convexity of $\uf^*$) to the above optimization, where $\ux_{k+1} = (\ux_{k} + \eta_u x_{k+1})/(1+\eta_u)$. Hence, we prove the equivalence between the rules by induction.
\end{proof}

\subsection{Extension of LPD to a problem with additional proximal-friendly terms}
\label{sec:lpd_extension}
LPD can be extended to solve more general (possibly nonsmooth) minimax problems with the same guarantees:
\begin{align}\label{eq:scsc_general_problem_appendix}
\min_{x \in \cX} \max_{y \in \cY} \; F(x) + f(x) + \Ip{y}{Ax} - h(y) - H(x),
\end{align}
where $F$ and $H$ are convex (and possibly non-smooth) and we have access to their proximal operators and $f,h$ satisfy Assumption \ref{assume:ldp_main_assumptions}. The only change we need to make is to replace the $x_{k+1}$ and $y_{k+1}$ update steps in Algorithm \ref{algo:general_lpd} with
\begin{align}
&x_{k+1} = \arg\min_{x \in \cX} \Ip{A^\top \ty_{k+1} + \tu_{k+1}}{x} + \nonumber \\
&\;\;\;\;\;\;\;\;\;\;\;\;\;\;\;\;\;\;\;\;\;\;\;\;\; \|x-x_{k}\|^2/{2\eta_{x}} + {\mu_x}\|x\|^2/2 + F(x) \nonumber\\
&y_{k+1} = \arg\min_{y \in \cY} -\Ip{A^\top \tx_{k+1} + \tv_{k+1}}{y} + \label{eq:scsc_lpd_rule_proximal}  \\
&\;\;\;\;\;\;\;\;\;\;\;\;\;\;\;\;\;\;\;\;\;\;\;\;\; \|y-y_{k}\|^2/{2\eta_{y}} + {\mu_y}\|y\|^2/2 + H(x) \nonumber \,.
\end{align}
Then the same guarantees as Corollaries \ref{cor:smooth_scsc_primal_dual} and \ref{cor:smooth_csc_primal_dual} holds for this update. We omit the analysis since it is similar to the proof of Theorem \ref{thm:smooth_lpd_primal_dual}.

\subsection{Mirror-Descent lemma}
\label{sec:md_lemma}

\begin{lemma}[\cite{nesterov2018lectures}]\label{lem:md_lemma}
Let $r$ be strongly convex, $F$ be $\mu$-(relatively) strongly w.r.t.~to $r$, and 
\begin{align}
x_{k+1} = \arg\min_{x} \Ip{g}{x} + F(x) + \frac1{\eta} V^{r}_{x_k}(x)
\end{align}
then
\begin{align}
\Ip{g}{x_{k+1}-x} + F(x_{k+1}) - F(x) \leq \frac1{\eta} V^{r}_{x_k}(x) - (\frac1{\eta} + \mu) V^{r}_{x_{k+1}}(x) - \frac1{\eta} V^{r}_{x_k}(x_{k+1})
\end{align}
\end{lemma}
\section{Algorithm for Bilineraly-coupled smooth minimax problem}
First we will prove a general result for Bilineraly-coupled smooth minimax problem. Then we specialize it to the Bi-SC-SC and Bi-C-SC cases. 

As mentioned in the main text we first apply the follow reformulation to \eqref{eq:scsc_problem}.
\begin{align}
\min_{x \in \cX} \max_{y \in \cY} \; &[g(x, y) = f(x) + \Ip{y}{Ax} - h(y)] \label{eq:scsc_problema_pf} \\
= \min_{x \in \cX} \max_{y \in \cY} \min_{v} \max_{u} &[g(x,y;u,v) = -\uf^*(u) + \Ip{u}{x} + \frac{\mu_x}2 \|x\|^2 + \Ip{y}{Ax} - \frac{\mu_y}2 \|y\|^2 - \Ip{v}{y} + \uh^*(v)]
\end{align}
where
\begin{align}
\uf^*(u) := \max_{x \in \cX \, \cap \, \domain(f)} \Ip{u}{x} - [\uf := f(x) - (\mu_x/2) \|x\|^2] \\
\uh^*(v) := \max_{y \in \cY \,\cap\, \domain(h)} \Ip{v}{y} - [\uh := h(x) - (\mu_y/2) \|y\|^2]
\end{align}
Note that by Lemma \ref{lem:sc_minus_quad}, $\uf$ is convex and $(L_x - \mu_x)$-smooth, and $\uh$ is convex and $(L_y - \mu_y)$-smooth. Then by Lemma \ref{lem:fenchel_props}(a) $\uf^*$ is $1/(L_x - \mu_x)$-strongly convex, and $\uh^*$ is $1/(L_y - \mu_y)$-strongly convex.

Instead of analyzing the Algorithm \ref{algo:lpd}, we analyze the original update rule \eqref{eq:scsc_lpd_rule} (Algorithm \ref{algo:general_lpd}) which is a conceptually easier implementation of LPD. By the following lemma we show that Algorithm \ref{algo:lpd} and Algorithm \ref{algo:general_lpd} of these are equivalent, when initialized appropriately.

\begin{lemma}[Same as Lemma \ref{lem:scsc_lpd_update}] \label{lem:scsc_lpd_update_appendix}
Let us initialize Algorithm \ref{algo:lpd} with $(\ux_{-1},\ux_{0}, \uy_{-1},\uy_{0})$$\,=\,$$(x_0, x_0, y_0, y_0)$, Algorithm \ref{algo:general_lpd} with $(u_{-1},u_{0}, v_{-1},v_{0})$$\,=\,$$(\nabla \uf(\ux_{-1}),\nabla \uf(\ux_{0}),\nabla \uh(\ux_{-1}),\nabla \uh(\ux_{0}))$, and both the algorithms with the same $(x_0, y_0)$. Then for problem \eqref{eq:scsc_problema_pf}, iterates $(x_{k}, y_{k})$ of the Algorithm \ref{algo:lpd} and Algorithm \ref{algo:general_lpd} are the same.
\end{lemma}
\begin{proof}
We omit the proof since we can easily prove it using the same techniques as used in the proof of Lemma \ref{lem:sc_agd_update}.
\end{proof}

\begin{algorithm}[t!]
	\SetAlgoLined
	\DontPrintSemicolon
	\SetKwProg{myproc}{Procedure}{}{}
	{\bf Required}: $\cX$, $\cY$, $(f, L_x, \mu_x)$, $(A, \|A\|)$, $(h, L_y, \mu_y)$, %
	$K$, %
	$\{(\eta_{x,k}, \eta_{y,k}, \eta_{u,k}, \eta_{v,k}, \theta_{k})\}_{k=0}^{K-1}$ \\
	\nl Initialize $(x_{-1}, y_{-1}) = (x_0, y_0) \in \cX \times \cY$ \\
	\nl Set $\uf = f - (\mu_x/2) \|\cdot\|^2$, $\uh = h - (\mu_y/2) \|\cdot\|^2$, %
	$(\ux_{-1}, \uy_{-1}) = (\ux_0, \uy_0) = (x_0, y_0)$ \\
	\For{$0\leq k\leq K-1$}{ 
		\nl $\tx_{k+1} = x_{k} + \theta_k (x_k - x_{k-1})\,,\;
        \ty_{k+1} = y_{k} + \theta_k (y_k - y_{k-1})$, \\
        $\tu_{k+1} = u_{k} + \theta_k (u_{k} - u_{k-1})$, $\tv_{k+1}) = v_{k} + \theta_k (v_{k} - v_{k-1})$
        \label{algo_line:extrapolate} \\
		\nl $x_{k+1} = \arg\min_{x \in \cX} \Ip{A^\top \ty_{k+1} + \tu_{k+1} }{x} + \frac1{2\eta_{x,k}} \|x-x_{k}\|^2 + \frac{\mu_x}2\|x\|^2$ 
		\label{algo_line:x_update} \\
		\nl $y_{k+1} = \arg\min_{y \in \cY} -\Ip{A \tx_{k+1} - \tv_{k+1}}{y} + \frac1{2\eta_{y,k}} \|y-y_{k}\|^2 + \frac{\mu_{y}}2 \|y\|^2$ 
		\label{algo_line:y_update} \\
        \nl $u_{k+1} = \arg\min_{u} -\Ip{x_{k+1}}{u} + \uf^*(u) + V^{\uf^*}_{u_k}(u)/{\eta_u}$ \label{algo_line:u_update} \\
        \nl $v_{k+1} = \arg\min_{v} -\Ip{y_{k+1}}{v} + \uh^*(v) + V^{\uh^*}_{v_k}(v)/{\eta_v}$ \label{algo_line:v_update} 
	}
	\nl \Return $(x_K, y_K, u_K, v_K)$
	\caption{O-LPD: Original Lifted Primal-Dual algorithm}
	\label{algo:general_lpd}
\end{algorithm}
We prove the follow Theorem for characterizing the output of Algorithm \ref{algo:general_lpd}.
\begin{theorem}\label{thm:smooth_lpd_primal_dual}
Let there exists positive numbers $\lambda_k$, $\alpha_{x,k}$, $\alpha_{y,k}$, $\alpha_{u,k}$, $\alpha_{v,k}$ for all $k=-1,0,1\ldots$, such that $\lambda_{k-1} = \theta_k \lambda_k$,
\begin{align}
&\theta_k(\frac{\|A\| }{\alpha_{y,k}} + \frac{1}{\alpha_{u,k}}) + {\|A\| \alpha_{x,k+1}} \leq \frac{1}{\eta_{x,k}}  \;,\;\;\;\; \alpha_u \leq \frac1{\eta_{u,k} (L_x-\mu_x)} \;, \label{eq:smooth_scsc_primal_dual_alpha1} \\
&\theta_k(\frac{\|A\| }{\alpha_x} + \frac{1}{\alpha_v}) + {\|A\| \alpha_{y,k+1}} \leq \frac{1}{\eta_{y,k}}  \;,\;\;\;\;
\alpha_v \leq \frac1{\eta_{v,k} (L_y-\mu_y)} \label{eq:smooth_scsc_primal_dual_alpha2} \\
\frac{\lambda_{k+1}}{\lambda_{k}} \leq \min \bigg( &\frac{\eta_{x,k+1}(1+\eta_{x,k} \mu_x)}{\eta_{x,k}} \,,\; \frac{\eta_{y,k+1}(1+\eta_{y,k} \mu_x)}{\eta_{x,k}} \,,\;
\frac{\eta_{u,k+1}(1+\eta_{u,k} \mu_x)}{\eta_{x,k}} \,,\;
\frac{\eta_{v,k+1}(1+\eta_{v,k} \mu_x)}{\eta_{x,k}} \bigg) \label{eq:smooth_scsc_primal_dual_lambda}
\end{align}
for all $k=0,1,\ldots$.
Then the following is true for any $K = 1,2,\ldots$, $x \in \cX$, $y  \in \cY$, $u$, $v$
\begin{align}
\sum_{k=0}^{K-1} \lambda_k  &[\Phi(x_{k+1}, y; u, v_{k+1}) - \Phi(x, y_{k+1}; u_{k+1}, v)] \nonumber \\
\;\leq\; &\frac{\lambda_0}{2\eta_{x,0}}\|x-x_{0}\|^2 - \frac{\lambda_{K} \|A\| \alpha_{x,K+1}}{2} \|x-x_{K}\|^2 +
\frac{\lambda_0 }{2\eta_{y,0}}\|y-y_{0}\|^2 - \frac{\lambda_{K} \|A\| \alpha_{y,K+1}}{2} \|y-y_{K}\|^2 \;+ \nonumber \\
&\frac{\lambda_0}{\eta_{u,0}}V^{\uf^*}_{u_0}(u) - \frac{\lambda_{K}}{\eta_{u,K}}V^{\uf^*}_{u_{K}}(u) +
\frac{\lambda_0}{\eta_{v,0}}V^{\uh^*}_{v_{0}}(v) - \frac{\lambda_{K}}{\eta_{v,K}} V^{\uh^*}_{v_{K}}(v)
\end{align}
\end{theorem}
Note that proof of Theorem \ref{thm:smooth_lpd_primal_dual} closely follows the steps used in the proof of Lemma \ref{lem:ppm_rule}.
\begin{proof}
Let $x \in \cX$ and $y \in \cY$. Using Steps \ref{algo_line:x_update} and \ref{algo_line:y_update} (Algorithm \ref{algo:general_lpd}) and Lemma \ref{lem:md_lemma} twice---once with $g=A^\top \ty_{k+1} + \tu_{k+1}$, $F = (\mu_x/2)\|\cdot\|^2 + F$ and $r = \|\cdot\|^2/2$, and second time with $g=-A \tx_{k+1} + \tv_{k+1}$, $F = (\mu_y/2)\|\cdot\|^2 + H$ and $r = \|\cdot\|^2/2$---we get
\begin{align}
\Ip{A^\top \ty_{k+1} + \tu_{k+1}}{x_{k+1} - x} + &\frac{\mu_x}2 (\|x_{k+1}\|^2 - \|x_{}\|^2) + F(x_{k+1}) - F(x) \nonumber \\
&\leq \frac1{2\eta_{x,k}}(\|x-x_{k}\|^2 - \|x-x_{k+1}\|^2 - \|x_{k+1}-x_{k}\|^2)  - \frac{\mu_x}{2}\|x-x_{k+1}\|^2 \\
\Ip{-A \tx_{k+1} + \tv_{k+1}}{y_{k+1} - x} + &\frac{\mu_y}2 (\|y_{k+1}\|^2  - \|y_{}\|^2) + H(y_{k+1}) - H(y) \nonumber \\
&\leq \frac1{2\eta_{y,k}}(\|y-y_{k}\|^2 - \|y-y_{k+1}\|^2 - \|y_{k+1}-y_{k}\|^2) - \frac{\mu_y}{2}\|y-y_{k+1}\|^2
\end{align}
Note that $\uf^*$ and $\uh^*$ are $1$-strong convex w.r.t themselves.
Again using Step \ref{algo_line:u_update} (Algorithm \ref{algo:general_lpd}) and Lemma \ref{lem:md_lemma} twice---once with $g=-x_{k+1}$, $F = \uf^*$ and $r = \uf^*$, and second time time with $g=-y_{k+1}$, $F = \uh^*$ and $r = \uh^*$---we get
\begin{align}
\Ip{-x_{k+1}}{u_{k+1} - u} + \uf^*(u_{k+1}) - \uf^*(u) &\leq \frac1{\eta_{u,k}}(V^{\uf^*}_{u_k}(u) - V^{\uf^*}_{u_{k+1}}(u) - V^{\uf^*}_{u_{k}}(u_{k+1})) - V^{\uf^*}_{u_{k+1}}(u) \\
\Ip{y_{k+1}}{v_{k+1} - v} + \uh^*(v_{k+1}) - \uh^*(v) &\leq \frac1{\eta_{v,k}}(V^{\uh^*}_{v_{k}}(v) - V^{\uh^*}_{v_{k+1}}(v) - V^{\uh^*}_{v_{k}}(v_{k+1})) - V^{\uh^*}_{v_{k+1}}(v)
\end{align}
Adding the above four equations and using the definition $\gap_{z,w}(z_{k+1}, w_{k+1}) = \Phi(x_{k+1}, y; u, v_{k+1}) - \Phi(x, y_{k+1}; u_{k+1}, v)$, where $z = (x,y)$ and $w=(u,v)$
\begin{align}
\gap_{z,w}(z_{k+1}, w_{k+1}) \;=\; &\Phi(x_{k+1}, y; u, v_{k+1}) - \Phi(x, y_{k+1}, u_{k+1}, v) \nonumber \\
\;\leq\; &\frac1{2\eta_{x,k}}\|x-x_{k}\|^2 - (\frac1{2\eta_{x,k}} + \frac{\mu_x}{2})\|x-x_{k+1}\|^2 - \frac1{2\eta_{x,k}} \|x_{k+1}-x_{k}\|^2 \;+ \nonumber \\
&\frac1{2\eta_{y,k}}\|y-y_{k}\|^2 -  (\frac1{2\eta_{y,k}} + \frac{\mu_y}{2}) \|y-y_{k+1}\|^2 - \frac1{2\eta_{y,k}} \|y_{k+1}-y_{k}\|^2 \;+ \nonumber \\
&\frac1{\eta_{u,k}}V^{\uf^*}_{u_k}(u) - (\frac1{\eta_{u,k}}+1)V^{\uf^*}_{u_{k+1}}(u) - \frac1{\eta_{u,k}} V^{\uf^*}_{u_{k}}(u_{k+1}) \;+ \nonumber \\
&\frac1{\eta_{v,k}}(V^{\uh^*}_{v_{k}}(v) - (\frac1{\eta_{v,k}} +1)V^{\uh^*}_{v_{k+1}}(v) - \frac1{\eta_{v,k}}V^{\uh^*}_{v_{k}}(v_{k+1}) \;+ \nonumber \\
&\Ip{y_{k+1} - \ty_{k+1}}{A(x_{k+1} - x)} \;+ %
-\Ip{y_{k+1} - y}{A(x_{k+1} - \tx_{k+1})} \;+ \nonumber \\
&\Ip{u_{k+1} - \tu_{k+1}}{x_{k+1} - x} \;+ 
\Ip{v_{k+1} - \tv_{k+1}}{y_{k+1} - y} 
\label{eq:scsc_lpd_pf_eq1}
\end{align}
We can further expand out the last four term in the above inequality as follows. Using Step \ref{algo_line:extrapolate} (Algorithm \ref{algo:general_lpd}) and Cauchy-Schwarz inequality we get
\begin{align}
\Ip{y_{k+1} - \ty_{k+1}}{A(x_{k+1} - x)} &= \theta_k \Ip{y_{k-1} - y_{k}}{A(x_{k} - x)} - \Ip{y_{k} - y_{k+1}}{A(x_{k+1} - x)} \;+ \nonumber \\
&\;\;\;\;\;\; \theta_k \Ip{y_{k-1} - y_{k}}{A(x_{k+1} - x_{k})} \nonumber \\
&\leq \theta_k \Ip{y_{k-1} - y_{k}}{A(x_{k} - x)} - \Ip{y_{k} - y_{k+1}}{A(x_{k+1} - x)} \;+ \nonumber \\
&\;\;\;\;\;\; \frac{\theta_k \|A\| \alpha_{y,k}}2 \|y_{k-1} - y_{k}\|^2 + \frac{\theta_k \|A\| }{2\alpha_{y,k}} \|x_{k+1} - x_{k}\|^2
\label{eq:scsc_lpd_pf_eq2a}
\end{align}
for some $\alpha_{y,k} \geq 0$. Similarly we can show that
\begin{align}
-\Ip{y_{k+1} - y}{A(x_{k+1} - \tx_{k+1})}
&\leq -\theta_k \Ip{y_{k} - y}{A(x_{k-1} - x_{k})} + \Ip{y_{k+1} - y}{A(x_{k} - x_{k+1})} \;+ \nonumber \\
&\;\;\;\;\;\; \frac{\theta_k \|A\| \alpha_{x,k}}2 \|x_{k-1} - x_{k}\|^2 + \frac{\theta_k \|A\| }{2\alpha_{x,k}} \|y_{k+1} - y_{k}\|^2 \label{eq:scsc_lpd_pf_eq2b}
\\
\Ip{u_{k+1} - \tu_{k+1}}{x_{k+1} - x}
&\leq \theta_k \Ip{u_{k-1} - u_{k}}{x_{k} - x} - \Ip{u_{k} - u_{k+1}}{x_{k+1} - x} \;+ \nonumber \\
&\;\;\;\;\;\; \frac{\theta_k \alpha_{u,k}}2 \|u_{k-1} - u_{k}\|^2 + \frac{\theta_k }{2\alpha_{u,k}} \|x_{k+1} - x_{k}\|^2 \label{eq:scsc_lpd_pf_eq2c}
\\
\Ip{v_{k+1} - \tv_{k+1}}{y_{k+1} - y} 
&\leq \theta_k \Ip{v_{k-1} - v_{k}}{y_{k} - y} - \Ip{v_{k} - v_{k+1}}{y_{k+1} - y} \;+ \nonumber \\
&\;\;\;\;\;\; \frac{\theta_k \alpha_{v,k}}2 \|v_{k-1} - v_{k}\|^2 + \frac{\theta_k }{2\alpha_{v,k}} \|y_{k+1} - y_{k}\|^2 \label{eq:scsc_lpd_pf_eq2d}
\end{align}
for some $\alpha_{x,k} \geq 0$, $\alpha_{u,k} \geq 0$, and $\alpha_{v,k} \geq 0$.
Using Lemma \ref{lem:breg_props}(a) and $1/(L_x - \mu_x)$- and $1/(L_y - \mu_y)$-strong convexity of $\uf^*$ and $\uh^*$, respectively we get that
\begin{align}
- \frac{1}{\eta_{u,k}} V^{\uf^*}_{u_{k}}(u_{k+1})\leq -\frac{1}{2\eta_{u,k} (L_x - \mu_x)} \|u_{k+1} - u_{k}\|^2 \label{eq:scsc_lpd_pf_eq3a} \\
- \frac{1}{\eta_{v,k}} V^{\uh^*}_{v_{k}}(v_{k+1})\leq -\frac{1}{2\eta_{v,k} (L_y - \mu_y)} \|v_{k+1} - v_{k}\|^2 \label{eq:scsc_lpd_pf_eq3b}
\end{align}

Summing equations \eqref{eq:scsc_lpd_pf_eq1}, \eqref{eq:scsc_lpd_pf_eq2a}, \eqref{eq:scsc_lpd_pf_eq2b}, \eqref{eq:scsc_lpd_pf_eq2c}, \eqref{eq:scsc_lpd_pf_eq2d}, \eqref{eq:scsc_lpd_pf_eq3a}, and \eqref{eq:scsc_lpd_pf_eq3b} up we get
\begin{align}
\gap_{z,w}(z_{k+1}, w_{k+1})
\;\leq\; &\frac{1}{2\eta_{x,k}}\|x-x_{k}\|^2 - (\frac1{2\eta_{x,k}} + \frac{\mu_x}{2}) \|x-x_{k+1}\|^2 +
\frac{1}{2\eta_{y,k}}\|y-y_{k}\|^2 - (\frac1{2\eta_{y,k}} + \frac{\mu_y}{2}) \|y-y_{k+1}\|^2 \;+ \nonumber \\
&\frac{1}{\eta_{u,k}}V^{\uf^*}_{u_k}(u) - (\frac1{\eta_{u,k}} + 1) V^{\uf^*}_{u_{k+1}}(u) +
\frac{1}{\eta_{v,k}}V^{\uh^*}_{v_{k}}(v) - (\frac1{\eta_{v,k}} + 1) V^{\uh^*}_{v_{k+1}}(v) \;+ \nonumber \\
&\theta_k \Ip{y_{k-1} - y_{k}}{A(x_{k} - x)} - \Ip{y_{k} - y_{k+1}}{A(x_{k+1} - x)} \;+ \nonumber \\
&-\theta_k \Ip{y_{k} - y}{A(x_{k-1} - x_{k})} + \Ip{y_{k+1} - y}{A(x_{k} - x_{k+1})} \;+ \nonumber \\
&\theta_k \Ip{u_{k-1} - u_{k}}{x_{k} - x} - \Ip{u_{k} - u_{k+1}}{x_{k+1} - x} \;+ \nonumber \\
&\theta_k \Ip{v_{k-1} - v_{k}}{y_{k} - y} - \Ip{v_{k} - v_{k+1}}{y_{k+1} - y} \;+ \nonumber \\
&\theta_k \frac{\|A\| \alpha_{x,k}}2 \|x_{k} - x_{k-1}\|^2 - (\frac{1}{2\eta_{x,k}} - \theta_k (\frac{\|A\| }{2\alpha_{y,k}} + \frac{1}{2\alpha_{u,k}})) \|x_{k+1} - x_{k}\|^2 \;+ \nonumber \\
&\theta_k \frac{\|A\| \alpha_{y,k}}2 \|y_{k} - y_{k-1}\|^2 - (\frac{1}{2\eta_{y,k}} - \theta_k (\frac{\|A\| }{2\alpha_{x,k}} + \frac{1 }{2\alpha_{v,k}})) \|y_{k+1} - y_{k}\|^2 \;+ \nonumber \\
&\theta_k \frac{\alpha_{u,k}}2 \|u_{k} - u_{k-1}\|^2 -\frac{1}{2\eta_{u,k} (L_x - \mu_x)} \|u_{k+1} - u_{k}\|^2 \;+ \nonumber \\
&\theta_k \frac{\alpha_{v,k}}2 \|v_{k} - v_{k-1}\|^2 -\frac{1}{2\eta_{v,k} (L_y - \mu_y)} \|v_{k+1} - v_{k}\|^2
\end{align}
Assuming
$\theta_k(\frac{\|A\| }{2\alpha_{y,k}} + \frac{1}{2\alpha_{u,k}}) + \frac{\|A\| \alpha_{x,k+1}}2 \leq \frac{1}{2\eta_{x,k}}$,
$\theta_k(\frac{\|A\| }{2\alpha_{x,k}} + \frac{1}{2\alpha_{v,k}}) + \frac{\|A\| \alpha_{y,k+1}}2 \leq \frac{1}{2\eta_{y,k}}$,
$\alpha_{u,k} \leq \frac1{\eta_{u,k+1} (L_x-\mu_x)}$, and $\alpha_{v,k} \leq \frac1{\eta_{v,k+1} (L_y-\mu_y)}$ we get
\begin{align}
\gap_{z,w}(z_{k+1}, w_{k+1})
\;\leq\; &\frac{1}{2\eta_{x,k}}\|x-x_{k}\|^2 - (\frac1{2\eta_{x,k}} + \frac{\mu_x}{2}) \|x-x_{k+1}\|^2 +
\frac{1}{2\eta_{y,k}}\|y-y_{k}\|^2 - (\frac1{2\eta_{y,k}} + \frac{\mu_y}{2}) \|y-y_{k+1}\|^2 \;+ \nonumber \\
&\frac{1}{\eta_{u,k}}V^{\uf^*}_{u_k}(u) - (\frac1{\eta_{u,k}} + 1) V^{\uf^*}_{u_{k+1}}(u) +
\frac{1}{\eta_{v,k}}V^{\uh^*}_{v_{k}}(v) - (\frac1{\eta_{v,k}} + 1) V^{\uh^*}_{v_{k+1}}(v) \;+ \nonumber \\
&\theta_k \Ip{y_{k-1} - y_{k}}{A(x_{k} - x)} - \Ip{y_{k} - y_{k+1}}{A(x_{k+1} - x)} \;+ \nonumber \\
&-\theta_k \Ip{y_{k} - y}{A(x_{k-1} - x_{k})} + \Ip{y_{k+1} - y}{A(x_{k} - x_{k+1})} \;+ \nonumber \\
&\theta_k \Ip{u_{k-1} - u_{k}}{x_{k} - x} - \Ip{u_{k} - u_{k+1}}{x_{k+1} - x} \;+ \nonumber \\
&\theta_k \Ip{v_{k-1} - v_{k}}{y_{k} - y} - \Ip{v_{k} - v_{k+1}}{y_{k+1} - y} \;+ \nonumber \\
&\theta_k \frac{\|A\| \alpha_{x,k}}2 \|x_{k} - x_{k-1}\|^2 - \frac{\|A\| \alpha_{x,k+1}}2 \|x_{k+1} - x_{k}\|^2 \;+ \nonumber \\
&\theta_k \frac{\|A\| \alpha_{y,k}}2 \|y_{k} - y_{k-1}\|^2 - \frac{\|A\| \alpha_{y,k+1}}2 \|y_{k+1} - y_{k}\|^2 \;+ \nonumber \\
&\theta_k \frac{\alpha_{u,k}}2 \|u_{k} - u_{k-1}\|^2 - \frac{\alpha_{u,k+1}}2 \|u_{k+1} - u_{k}\|^2 \;+ 
\nonumber \\
&\theta_k \frac{\alpha_{v,k}}2 \|v_{k} - v_{k-1}\|^2 - \frac{\alpha_{v,k+1}}2 \|v_{k+1} - v_{k}\|^2
\end{align}
Multiplying both sides with $\lambda_k$, and using $\theta_k \lambda_k = \lambda_{k-1}$ and
\begin{align}
\frac{\lambda_{k+1}}{\lambda_{k}} \leq \min \bigg( \frac{\eta_{x,k+1}(1+\eta_{x,k} \mu_x)}{\eta_{x,k}} \,,\; \frac{\eta_{y,k+1}(1+\eta_{y,k} \mu_x)}{\eta_{x,k}} \,,\;
\frac{\eta_{u,k+1}(1+\eta_{u,k} \mu_x)}{\eta_{x,k}} \,,\;
\frac{\eta_{v,k+1}(1+\eta_{v,k} \mu_x)}{\eta_{x,k}} \bigg)
\end{align}
we get
\begin{align}
\lambda_k \gap_{z,w}(z_{k+1}, w_{k+1})
\;\leq\; &\frac{\lambda_k}{2\eta_{x,k+1}}\|x-x_{k}\|^2 - \frac{\lambda_{k+1}}{2\eta_{x,k+1}} \|x-x_{k+1}\|^2 +
\frac{\lambda_k }{2\eta_{y,k}}\|y-y_{k}\|^2 - \frac{\lambda_{k+1}}{2\eta_{y,k+1}} \|y-y_{k+1}\|^2 \;+ \nonumber \\
&\frac{\lambda_k}{\eta_{u,k}}V^{\uf^*}_{u_k}(u) - \frac{\lambda_{k+1}}{\eta_{u,k+1}}V^{\uf^*}_{u_{k+1}}(u) +
\frac{\lambda_k}{\eta_{v,k}}V^{\uh^*}_{v_{k}}(v) - \frac{\lambda_{k+1}}{\eta_{v,k+1}} V^{\uh^*}_{v_{k+1}}(v) \;+ \nonumber \\
&\lambda_{k-1} \Ip{y_{k-1} - y_{k}}{A(x_{k} - x)} - \lambda_k \Ip{y_{k} - y_{k+1}}{A(x_{k+1} - x)} \;+ \nonumber \\
&-\lambda_{k-1} \Ip{y_{k} - y}{A(x_{k-1} - x_{k})} + \lambda_k \Ip{y_{k+1} - y}{A(x_{k} - x_{k+1})} \;+ \nonumber \\
&\lambda_{k-1} \Ip{u_{k-1} - u_{k}}{x_{k} - x} - \lambda_k \Ip{u_{k} - u_{k+1}}{x_{k+1} - x} \;+ \nonumber \\
&\lambda_{k-1} \Ip{v_{k-1} - v_{k}}{y_{k} - y} - \lambda_k \Ip{v_{k} - v_{k+1}}{y_{k+1} - y} \;+ \nonumber \\
&\lambda_{k-1} \frac{\|A\| \alpha_{x,k}}2 \|x_{k} - x_{k-1}\|^2 - \lambda_k \frac{\|A\| \alpha_{x,k+1}}2 \|x_{k+1} - x_{k}\|^2 \;+ \nonumber \\
&\lambda_{k-1} \frac{\|A\| \alpha_{y,k}}2 \|y_{k} - y_{k-1}\|^2 - \lambda_k \frac{\|A\| \alpha_{y,k+1}}2 \|y_{k+1} - y_{k}\|^2 \;+ \nonumber \\
&\lambda_{k-1} \frac{\alpha_{u,k}}2 \|u_{k} - u_{k-1}\|^2 - \lambda_k \frac{\alpha_{u,k+1}}2 \|u_{k+1} - u_{k}\|^2 \;+ \nonumber \\
&\lambda_{k-1} \frac{\alpha_{v,k}}2 \|v_{k} - v_{k-1}\|^2 - \lambda_k \frac{\alpha_{v,k+1}}2 \|v_{k+1} - v_{k}\|^2
\end{align}
Summing the iterations of the above inequality for $k=0,\ldots,K-1$ and without loss of generality setting $\lambda_{-1} = 0$, or $x_{-1} = x_0$, $y_{-1} = y_0$, $u_{-1} = u_0$, and $v_{-1} = v_0$
we get
\begin{align}
\sum_{k=0}^{K-1} \lambda_k \gap_{z,w}(z_{k+1}, w_{k+1})
\;\leq\; &\frac{\lambda_0}{2\eta_{x,0}}\|x-x_{0}\|^2 - \frac{\lambda_{K}}{2\eta_{x,K}} \|x-x_{K}\|^2 +
\frac{\lambda_0 }{2\eta_{y,0}}\|y-y_{0}\|^2 - \frac{\lambda_{K}}{2\eta_{y,K}} \|y-y_{K}\|^2 \;+ \nonumber \\
&\frac{\lambda_0}{\eta_{u,0}}V^{\uf^*}_{u_0}(u) - \frac{\lambda_{K}}{\eta_{u,K}}V^{\uf^*}_{u_{K}}(u) +
\frac{\lambda_0}{\eta_{v,0}}V^{\uh^*}_{v_{0}}(v) - \frac{\lambda_{K}}{\eta_{v,K}} V^{\uh^*}_{v_{K}}(v) \;+ \nonumber \\
&- \lambda_{K-1} \Ip{y_{K-1} - y_{K}}{A(x_{K} - x)} %
+ \lambda_{K-1} \Ip{y_{K} - y}{A(x_{K-1} - x_{K})} \;+ \nonumber \\
&-\lambda_{K-1} \Ip{u_{K-1} - u_{K}}{x_{K} - x}  %
-\lambda_{K-1} \Ip{v_{K-1} - v_{K}}{y_{K} - y} \;+ \nonumber \\
&- \lambda_{K-1} \frac{\|A\| \alpha_{x,K}}2 \|x_{K} - x_{K-1}\|^2 
- \lambda_{K-1} \frac{\|A\| \alpha_{y,K}}2 \|y_{K} - y_{K-1}\|^2
\;+ \nonumber \\
&- \lambda_{K-1} \frac{\alpha_{u,K}}2 \|u_{K} - u_{K-1}\|^2
- \lambda_{K-1} \frac{\alpha_{v,K}}2 \|v_{K} - v_{K-1}\|^2  \label{eq:scsc_lpd_pf_eq4}
\end{align}

Using Cauchy-Schwarz inequality we can show that
\begin{align}
-\lambda_{K-1} \Ip{y_{K-1} - y_{K}}{A(x_{K} - x)} \leq 
\frac{\lambda_{K-1} \|A\|\alpha_{y,K}}{2} \|y_{K-1} - y_{K}\|^2 + \frac{\lambda_{K-1} \|A\|}{2\alpha_{y,K}} \|x_{K} - x\|^2 \label{eq:scsc_lpd_pf_eq5a} \\
\lambda_{K-1} \Ip{y_{K} - y}{A(x_{K-1} - x_{K})} \leq \frac{\lambda_{K-1} \|A\| \alpha_{x,K}}2 \|x_{K-1} - x_{K}\|^2 + \frac{\lambda_{K-1} \|A\|}{2\alpha_{x,K}} \|y_{K} - y\|^2 \label{eq:scsc_lpd_pf_eq5b} \\
-\lambda_{K-1} \Ip{u_{K-1} - u_{K}}{x_{K} - x} \leq 
\frac{\lambda_{K-1} \alpha_{u,K}}2 \|u_{K-1} - u_{K}\|^2 + \frac{\lambda_{K-1} \|A\|}{2\alpha_{u,K}} \|x_{K} - x\|^2 \label{eq:scsc_lpd_pf_eq5c} \\
-\lambda_{K-1} \Ip{v_{K-1} - v_{K}}{y_{K} - y} \leq 
\frac{\lambda_{K-1} \alpha_{v,K}}2 \|v_{K-1} - v_{K}\|^2 + \frac{\lambda_{K-1} \|A\|}{2\alpha_{v,K}} \|y_{K} - y\|^2 \label{eq:scsc_lpd_pf_eq5d}
\end{align}

Summing equations \eqref{eq:scsc_lpd_pf_eq4}, \eqref{eq:scsc_lpd_pf_eq5a}, \eqref{eq:scsc_lpd_pf_eq5b}, \eqref{eq:scsc_lpd_pf_eq5c}, and \eqref{eq:scsc_lpd_pf_eq5d} and then using 
$\theta_K \lambda_K = \lambda_{K-1}$, $\theta_K(\frac{\|A\| }{2\alpha_{y,K}} + \frac{1}{2\alpha_{u,K}}) + \frac{\|A\| \alpha_{x,K+1}}2 \leq \frac{1}{2\eta_{x,K}}$, and
$\theta_K(\frac{\|A\| }{2\alpha_{x,K}} + \frac{1}{2\alpha_{v,K}}) + \frac{\|A\| \alpha_{y,K+1}}2 \leq \frac{1}{2\eta_{y,K}}$,
we get
\begin{align}
\sum_{k=0}^{K-1} \lambda_k \gap_{z,w}(z_{k+1}, w_{k+1})
\;\leq\; &\frac{\lambda_0}{2\eta_{x,0}}\|x-x_{0}\|^2 - \frac{\lambda_{K} \|A\| \alpha_{x,K+1}}{2} \|x-x_{K}\|^2 +
\frac{\lambda_0 }{2\eta_{y,0}}\|y-y_{0}\|^2 - \frac{\lambda_{K} \|A\| \alpha_{y,K+1}}{2} \|y-y_{K}\|^2 \;+ \nonumber \\
&\frac{\lambda_0}{\eta_{u,0}}V^{\uf^*}_{u_0}(u) - \frac{\lambda_{K}}{\eta_{u,K}}V^{\uf^*}_{u_{K}}(u) +
\frac{\lambda_0}{\eta_{v,0}}V^{\uh^*}_{v_{0}}(v) - \frac{\lambda_{K}}{\eta_{v,K}} V^{\uh^*}_{v_{K}}(v)
\end{align}
\end{proof}

\section{Guarantee for Bi-SC-SC problem}
In this section we provide a guarantee for the output of Algorithm \ref{algo:lpd} in the Bi-SC-SC setting. We do this by specializing Theorem \ref{thm:smooth_lpd_primal_dual} to this case.
\begin{corollary}[Formal version of Theorem \ref{thm:smooth_scsc_primal_dual_informal}]
\label{cor:smooth_scsc_primal_dual}

Let $\overline{x}_0 = \overline{x}_{-1} = x_0$ and $\overline{y}_0 = \overline{y}_{-1} = y_0$. Additionally assume that $\eta_{x,k} = \eta_x$, $\eta_{y,k} = \eta_y$, $\eta_{u,k} = \eta_u$, $\eta_{v,k} = \eta_v$, and $\theta_k = \theta$ for all $k=0,1,\ldots$. If we set
\begin{align}
\kappa &= \sqrt{\frac{L_x}{\mu_x}-1} + \frac{2\|A\|}{\sqrt{\mu_x \mu_y}}+ \sqrt{\frac{L_y}{\mu_y}-1} \text{ , and }\\
\eta_x = \frac1{\mu_x}(\sqrt{\frac{L_x}{\mu_x}-1} + \frac{2\|A\|}{\sqrt{\mu_x \mu_y}})^{-1} 
\,,\,
\eta_y &= \frac1{\mu_y}(\sqrt{\frac{L_y}{\mu_y}-1} + \frac{2\|A\|}{\sqrt{\mu_x \mu_y}})^{-1}
\,,\,
\eta_u = (\sqrt{\frac{L_x}{\mu_x}-1})^{-1} 
\,,\,
\eta_v = (\sqrt{\frac{L_y}{\mu_y}-1})^{-1}
\end{align}
then for any $K > 0$,
we can show that
\begin{align}
\frac{\|A\|}{\sqrt{\mu_x \mu_y}} \big(\frac{\mu_x}{2} &\|x^*-x_{K}\|^2 +
\frac{\mu_y}{2} \|y^*-y_{K}\|^2 \big) + %
\nonumber \\
\;\leq\;
&\exp(- \frac{(K-1)}{(\kappa+1)})\big((\frac{1}{2\eta_x}+\frac{L_x-\mu_x}{2\eta_u})\|x^*-x_{0}\|^2 + (\frac{1}{2\eta_y}+\frac{L_y-\mu_y}{2\eta_v})\|y^*-y_{0}\|^2 \big)
\end{align}
\end{corollary}

\begin{proof}
We will first verify the parameter choices satisfies the required conditions of Theorem \ref{thm:smooth_lpd_primal_dual} for some choice of $\lambda_k$, $\alpha_{x,k}$, $\alpha_{y,k}$, $\alpha_{u,k}$, $\alpha_{v,k}$ for $k=-1,0,1,\ldots$. 

Let $\lambda_k = \gamma^k$ and $\theta_k = 1/\gamma$ where 
\begin{align}
\gamma &= 1 + \kappa^{-1}\;,\;\; \kappa = \sqrt{\frac{L_x}{\mu_x}-1} + \frac{2\|A\|}{\sqrt{\mu_x \mu_y}}+ \sqrt{\frac{L_y}{\mu_y}-1}
\end{align}
Clearly $\lambda_{k-1} = \theta \lambda_k$.
Next we will verify \eqref{eq:smooth_scsc_primal_dual_lambda} which simplifies to the 
\begin{align}
\gamma \leq 1+\min(\eta_x \mu_x, \eta_y \mu_y, \eta_u, \eta_v)
\end{align}
under our choice of $\lambda_k$ and $k$ invariant stepsize choices.
It is easy to see that
\begin{align}
\gamma = 1 + (\sqrt{\frac{L_x-\mu_x}{\mu_x}} + \frac{2\|A\|}{\sqrt{\mu_x \mu_y}}+ \sqrt{\frac{L_y-\mu_y}{\mu_y}})^{-1} \leq 1+ (\sqrt{\frac{L_x-\mu_x}{\mu_x}} + \frac{2\|A\|}{\sqrt{\mu_x \mu_y}})^{-1} = 1 +\mu_x \eta_x
\end{align}
Similarly we can also show that $\gamma \leq 1+ \min(\eta_y \mu_y, \eta_u, \eta_v)$.

Let $\alpha_{x,k}$, $\alpha_{y,k}$, $\alpha_{u,k}$, $\alpha_{v,k}$ be invariant to $k$ and $\alpha_{x,k} = \sqrt{\frac{\mu_x}{\mu_y}}$, $\alpha_{y,k} = \sqrt{\frac{\mu_y}{\mu_x}}$, $\alpha_{u,k} = \frac1{\sqrt{({L_x-\mu_x}){\mu_x}}}$, $\alpha_{v,k} = \frac1{\sqrt{({L_y-\mu_x}){\mu_y}}}$ for all $k=0,1,\ldots$. 
Next we verify conditions \eqref{eq:smooth_scsc_primal_dual_alpha1} and \eqref{eq:smooth_scsc_primal_dual_alpha2}.
We can show that
\begin{align}
\alpha_{u,k} = \frac1{\sqrt{({L_x-\mu_x}){\mu_x}}} \leq \frac1{\sqrt{({L_x-\mu_x}){\mu_x}}} = \frac1{\eta_{u,k} (L_x-\mu_x)}
\end{align}
and
\begin{align}
&\theta_k (\frac{\|A\| }{\alpha_{y,k}} + \frac{1}{\alpha_{u,k}}) + {\|A\| \alpha_{x,k+1}} = \frac{\mu_x}{\gamma}(\frac{\|A\|}{\sqrt{\mu_x \mu_y}} + \sqrt{\frac{L_x}{\mu_x}-1}) + \mu_x \frac{\|A\|}{\sqrt{\mu_x \mu_y}} < \mu_x (\sqrt{\frac{L_x}{\mu_x}-1} + \frac{2\|A\|}{\sqrt{\mu_x \mu_y}})= \frac{1}{\eta_{x,k}} 
\end{align}
Similar we can also show that,
\begin{align}
&\theta_k (\frac{\|A\| }{\alpha_{x,k}} + \frac{1}{\alpha_{v,k}}) + {\|A\| \alpha_{y,k+1}} \leq \frac{1}{\eta_{y,k}}  \;,\;\;\;\;
\alpha_{v,k} \leq \frac1{\eta_{v,k} (L_y-\mu_y)}
\end{align}

Then according to Theorem \ref{thm:smooth_lpd_primal_dual}, for any $K \geq 0$, $x \in \cX$, $y  \in \cY$, $u$, $v$
\begin{align}
\sum_{k=0}^{K-1} \gamma^k  &[\Phi(x_{k+1}, y; u, v_{k+1}) - \Phi(x, y_{k+1}; u_{k+1}, v)] \nonumber \\
\;\leq\; &\frac{1}{2\eta_{x}}\|x-x_{0}\|^2 - \sqrt{\frac{\mu_x}{\mu_y}}\frac{\gamma^{K} \|A\|}{2} \|x-x_{K}\|^2 +
\frac{1}{2\eta_{y}}\|y-y_{0}\|^2 - \sqrt{\frac{\mu_y}{\mu_x}} \frac{\gamma^{K} \|A\|}{2} \|y-y_{K}\|^2 \;+ \nonumber \\
&\frac{1}{\eta_{u}}V^{\uf^*}_{u_0}(u) - \frac{\gamma^K}{\eta_{u}}V^{\uf^*}_{u_{K}}(u) +
\frac{1}{\eta_{v}}V^{\uh^*}_{v_{0}}(v) - \frac{\gamma^K}{\eta_{v}} V^{\uh^*}_{v_{K}}(v)
\label{eq:smooth_scsc_primal_dual_cor_eq1}
\end{align}

Setting $x = x^*$, $y = y^*$, $u=u^{*} = \nabla \uf(x^*) = \arg\min_u \Ip{x}{u} - \uf^*(u)$, $v=v^{*} = \nabla \uh(y^*) = \arg\min_v \Ip{y}{v} - \uh^*(v)$ in \eqref{eq:smooth_scsc_primal_dual_cor_eq1} we get
\begin{align}
\sum_{k=0}^{K-1} \gamma^k  &[\Phi(x_{k+1}, y^*; u^*, v_{k+1}) - \Phi(x^*, y_{k+1}; u_{k+1}, v^*)] \nonumber \\
\;\leq\; &\frac{1}{2\eta_{x}}\|x^*-x_{0}\|^2 - \sqrt{\frac{\mu_x}{\mu_y}}\frac{\gamma^{K} \|A\|}{2} \|x^*-x_{K}\|^2 +
\frac{1}{2\eta_{y}}\|y^*-y_{0}\|^2 - \sqrt{\frac{\mu_y}{\mu_x}} \frac{\gamma^{K} \|A\|}{2} \|y^*-y_{K}\|^2 \;+ \nonumber \\
&\frac{1}{\eta_{u}}V^{\uf^*}_{u_0}(u^*) - \frac{\gamma^K}{\eta_{u}}V^{\uf^*}_{u_{K}}(u^*) +
\frac{1}{\eta_{v}}V^{\uh^*}_{v_{0}}(v^*) - \frac{\gamma^K}{\eta_{v}} V^{\uh^*}_{v_{K}}(v^*)
\end{align}
Notice that the LHS above is positive, since by Lemma \ref{lem:minimax_pseudo_gap}, $\Phi(x_{k+1}, y^*; u, v_{k+1}) - \Phi(x^*, y_{k+1}; u_{k+1}, v)
\geq 0$ for all $k=0,1,\ldots$. Then using this fact and Lemma \ref{lem:dual_bregman_div} four times, we get that
\begin{align}
\frac{\|A\|}{\sqrt{\mu_x \mu_y}} \big(\frac{\mu_x}{2} &\|x^*-x_{K}\|^2 +
\frac{\mu_y}{2} \|y^*-y_{K}\|^2 \big) + \frac{1}{2\eta_{u}}\frac{\|\nabla \uf(x^*)
 - \nabla \uf(\ux_K)\|^2}{ (L_x-\mu_x)} +
\frac{1}{2\eta_{v}} \frac{\|\nabla \uh(y^*)
 - \nabla \uh(\uy_K)\|^2}{ (L_y-\mu_y)}
\nonumber \\
\;\leq\;
&\gamma^{-K}\big((\frac{1}{2\eta_x}+\frac{L_x-\mu_x}{2\eta_u})\|x^*-x_{0}\|^2 + (\frac{1}{2\eta_y}+\frac{L_y-\mu_y}{2\eta_v})\|y^*-y_{0}\|^2 \big)
\end{align}
Using $1 - x \leq \exp(-x)$ we get
\begin{align}
\gamma^{-K}  = (\frac1{1 + \kappa^{-1}})^{K}  \leq (1 - \frac1{\kappa+1})^{K} 
\leq \exp(- \frac{K}{\kappa+1})
\end{align}
Combining above two inequality gives us the desired result.
\end{proof}

\section{Guarantee for Bi-C-SC problem}
In this section we provide a guarantee for the output of Algorithm \ref{algo:lpd} in the Bi-C-SC setting. We do this by specializing Theorem \ref{thm:smooth_lpd_primal_dual} to this case.

\begin{corollary}[Formal version of Theorem \ref{thm:smooth_csc_primal_dual_informal}]
\label{cor:smooth_csc_primal_dual}
Let $\overline{x}_0 = \overline{x}_{-1} = x_0$ and $\overline{y}_0 = \overline{y}_{-1} = y_0$ and
\begin{align}
\frac1{\eta_{x,k}} = \frac1{(k+1)\eta_x}\,,\,
\frac1{\eta_x} = 2L_x + \frac{16\|A\|^2}{\mu_y}
\,,\,
\frac1{\eta_{y,k}} = \frac1{(k+1)\eta_y} + \frac{k\mu_y}2\,,\,
\frac1{\eta_y} &= 2(L_y-\mu_y)
\,,\,
\eta_{u,k} = \frac2{k}
\,,\,
\eta_{v,k} = \frac2{k}\,.
\end{align}
Let $D_\cX = \max_{x \in \cX \cap \domain(f)} \|x-x_{0}\|$ and $D_\cY = \max_{y \in \cY \cap \domain(h)} \|y-y_{0}\|$. Then for any $K > 0$, \\
\noindent
$(a)$ if $D_\cX < \infty$ and $D_\cY < \infty$,
\begin{align}
\max_{y \in \cY} \phi(\overline{x}_{K}, {y}) - \min_{x \in \cX} \phi({x}, \overline{y}_{K}) 
&\leq \frac{2 L_x}{K(K+1)} D_\cX^2 + \frac{16\|A\|^2}{\mu_y K(K+1)} D_\cX^2 +
\frac{2(L_y-\mu_y)}{K(K+1)} D_\cY^2
\end{align}
where $(\overline{x}_K, \overline{y}_K) := \frac2{K(K+1)} {\sum_{k=1}^{K} k ( x_k, y_k) } = (\ux_K, \uy_K)$. 
\\
\noindent
$(b)$ even if the domain is unbounded we can show that
\begin{align}
\frac{\mu_y}{4}\|y^*-y_{K}\|^2 &\leq \frac{ 4L_x}{K(K+1)}\|x^*-x_{0}\|^2+ \frac{16\|A\|^2}{\mu_y K(K+1)}\|x^*-x_{0}\|^2+
\frac{4(L_y-\mu_y)}{K(K+1)} \|y^*-y_{0}\|^2
\end{align}
where $(\ux_K, \uy_K) = \frac2{K(K+1)}\sum_{k=1}^{K} k \cdot (x_{k}, y_{k}) = (\overline{y}_K, \overline{y}_K)$.
\\
$(c)$ if $\phi_p(x)$$\,=\,$$\max_{y \in \cY} \phi(x, y)$, and we do a warm restart on variable $y$ using $K_0^p = \Omega_{\varepsilon}(1)$ initial additional iterations of the same algorithm, then
\begin{align}
&\phi_p(\overline{x}_{K}) - \phi_p(x^*)
\leq  (\frac{4L_x}{K(K+1)} + \frac{32\|A\|^2}{\mu_y K(K+1)} + \frac{(L_y - \mu_y)}{\mu_y}\frac{8\|A\|^2}{\mu_y K(K+1)} )\|x^*-x_{0}\|^2\,.
\end{align}
$(d)$ if $D_\cX < \infty$ and $\phi_d(x)$$\,=\,$$\min_{x \in \cX} \phi(x, y)$, and we do a warm restart on variable $y$ with $K_0^d = \Omega_{\varepsilon}(1)$ initial additional iterations of the same algorithm, then
\begin{align}
\phi_d(y^*) - \phi_d(\overline{y}_{K})
\leq \frac{4L_x}{K(K+1)} D_\cX^2 + \frac{32\|A\|^2}{\mu_y K(K+1)}D_\cX^2\,.
\end{align}
\end{corollary}

\begin{proof}
We will first verify the parameter choices satisfies the required conditions of Theorem \ref{thm:smooth_lpd_primal_dual} for some choice of $\lambda_k$, $\alpha_{x,k}$, $\alpha_{y,k}$, $\alpha_{u,k}$, $\alpha_{v,k}$ for $k=-1,0,1,\ldots$. 

Let $\lambda_k = (k+1)$ and $\theta_k = k/(k+1)$. Clearly $\lambda_{k-1} = \theta \lambda_k$.
Next we will verify \eqref{eq:smooth_scsc_primal_dual_lambda} which simplifies to the 
\begin{align}
\frac{k+2}{\eta_{x,k+1}} \leq \frac{k+1}{\eta_{x,k}} \,,\;
\frac{k+2}{\eta_{y,k+1}} \leq \frac{(k+1)}{\eta_{y,k}} + (k+1) \mu_y \,,\; 
\frac{k+2}{\eta_{u,k+1}} \leq \frac{(k+1)}{\eta_{u,k}} + (k+1) \,, \text{ and } 
\frac{k+2}{\eta_{v,k+1}} \leq \frac{(k+1)}{\eta_{v,k}} + (k+1) 
\end{align}
under our choice of $\lambda_k$ and $\mu_x = 0$.
It is easy to verify that
\begin{align}
\frac{k+1}{\eta_{x,k}} = \frac1{\eta_x} &\geq \frac1{\eta_x} =  \frac{k+2}{\eta_{x,k+1}} \nonumber \\
\frac{(k+1)}{\eta_{y,k}} + (k+1) \mu_y = \frac1{\eta_y} + \frac{k(k+1) \mu_y}2 + (k+1) \mu_y &\geq  \frac1{\eta_y} + \frac{(k+1)(k+2) \mu_y}2 = \frac{k+2}{\eta_{y,k+1}} \nonumber \\
\frac{(k+1)}{\eta_{u,k}} + (k+1) = \frac{k(k+1)}2 + (k+1) &\geq  \frac{(k+1)(k+2}2 = \frac{k+2}{\eta_{u,k+1}} \nonumber \\
\frac{(k+1)}{\eta_{v,k}} + (k+1) = \frac{k(k+1)}2 + (k+1) &\geq  \frac{(k+1)(k+2}2 = \frac{k+2}{\eta_{v,k+1}} \nonumber
\end{align}
Let $\alpha_{x,k} = \frac{4\|A\|}{(k+1)\mu_y}$, $\alpha_{y,k} = \frac{k\mu_y}{4\|A\|}$, $\alpha_{u,k} = \frac{k}{2 L_x}$, $\alpha_{v,k} = \frac{k}{2 (L_y-\mu_y)}$ for all $k=0,1,\ldots$. 

Next we verify conditions \eqref{eq:smooth_scsc_primal_dual_alpha1} and \eqref{eq:smooth_scsc_primal_dual_alpha2}.
We can show that
\begin{align}
\alpha_{u,k} \leq  \frac{k}{2 L_x} \leq  \frac{k}{2 L_x}  = \frac1{\eta_{u,k} L_x}
\end{align}
and
\begin{align}
&\theta_k (\frac{\|A\| }{\alpha_{y,k}} + \frac{1}{\alpha_{u,k}}) + {\|A\| \alpha_{x,k+1}} = \frac{k}{k+1} (\frac{4\|A\|^2 }{k \mu_y} + \frac{2 L_x}{k}) + \frac{4\|A\|^2}{(k+2)\mu_y} \leq \frac{2L_x}{k+1} + \frac{16\|A\|^2}{\mu_y (k+1)} = \frac{1}{\eta_{x,k}} 
\end{align}
Similar we can also show that,
\begin{align}
\alpha_{v,k} \leq  \frac{k}{2 (L_y -\mu_y)} \leq  \frac{k}{2 (L_y -\mu_y)}  = \frac1{\eta_{v,k}(L_y -\mu_y)}
\end{align}
and
\begin{align}
&\theta_k (\frac{\|A\| }{\alpha_{x,k}} + \frac{1}{\alpha_{v,k}}) + {\|A\| \alpha_{y,k+1}} = \frac{k}{k+1} (\frac{(k+1)\mu_y }{4} + \frac{2 (L_x - \mu_y)}{k}) + \frac{k \mu_y}{4\|A\|} \leq \frac{2(L_y - \mu_y)}{(k+1)} + \frac{k\mu_y}2  =\frac{1}{\eta_{y,k}}
\end{align}

Then according to Theorem \ref{thm:smooth_lpd_primal_dual}, for any $K \geq 0$, $x \in \cX$, $y  \in \cY$, $u$, $v$
\begin{align}
\sum_{k=0}^{K-1} (k+1)  &[\Phi(x_{k+1}, y; u, v_{k+1}) - \Phi(x, y_{k+1}; u_{k+1}, v)] \nonumber \\
\;\leq\; &(L_x + \frac{8\|A\|^2}{\mu_y})\|x-x_{0}\|^2 - \frac{(K+1) 16 \|A\|^2 }{2 \mu_y (K+2)} \|x-x_{K}\|^2 +
(L_y-\mu_y)\|y-y_{0}\|^2 - (K+1)^2\frac{\mu_y}{8} \|y-y_{K}\|^2
\label{eq:smooth_csc_primal_dual_cor_eq1}
\end{align}

(a) We define that $(\overline{x}_K, \overline{y}_K; \overline{u}_K, \overline{v}_K) = (\sum_{k=1}^{K} (k+1))^{-1} {\sum_{k=1}^{K} (k+1) ( x_k, y_k; u_k, v_k)}$. Then $(\overline{x}_K, \overline{y}_K) = (\ux_K, \uy_K)$. Then $\overline{x}_K = \ux_K$ can be shown as follows
\begin{align}
\ux_K = \frac{\ux_{K-1} + \eta_x' x_{K}}{(1+\eta_x')} %
&= \frac{K-1}{K+1}\ux_{K-1} + \frac{2}{K+1}x_{K} \nonumber \\
&= \frac{(K-2)(K-1)}{K(K+1)}\ux_{K-2} + \frac{2(K-1)}{K(K+1)}x_{K-1} + \frac{2K}{K(K+1)}x_{K} \nonumber \\
&= \frac{(K-3)(K-2)}{K(K+1)}\ux_{K-3} + \frac{2(K-2)}{K(K+1)}\ux_{K-2} + \frac{2(K-1)}{K(K+1)}x_{K-1} + \frac{2(K)}{K(K+1)}x_{K} \nonumber \\
&\;\;\vdots \\
&= \frac2{K(K+1)}\sum_{k=1}^{K} k x_{k} = \overline{x}_K
\end{align}
Similarly, we can prove that $\overline{y}_K = \uy_K$. Then we can lower-bound the LHS of the \eqref{eq:smooth_csc_primal_dual_cor_eq1} using Jensen's inequality, convexity of $\Phi(\cdot, y; u, \cdot)$, and concavity of $\Phi(x, \cdot; \cdot, v)$ as follows.
\begin{align}
(\sum_{k=0}^{K-1} (k+1)) [\Phi(\overline{x}_{K}, {y}; {u}, \overline{v}_{K}) - \Phi({x}, \overline{y}_{K}; \overline{u}_{K}, v)]
\leq \sum_{k=0}^{K-1} (k+1)  &[\Phi(x_{k+1}, y; u, v_{k+1}) - \Phi(x, y_{k+1}; u_{k+1}, v)] 
\label{eq:smooth_csc_primal_dual_cor_eq2}
\end{align}
Notice that by Lemma \ref{lem:fenchel_props}(a), $\nabla f (x) = \arg\min_u \Ip{x}{u} - f^*(u)$ and $\nabla \uh (x) = \arg\min_v \Ip{y}{v} - \uh^*(v)$. Thus we have
\begin{align}
\phi(\overline{x}_{K}, {y}) - \phi({x}, \overline{y}_{K})
&= \min_{v} \max_{u} \Phi(\overline{x}_{K}, {y}; \nabla f(\overline{x}_{K}), v) - \Phi({x}, \overline{y}_{K};u, \nabla \uh(\overline{y}_{K})) \nonumber \\
&\leq \Phi(\overline{x}_{K}, {y}; \nabla f(\overline{x}_{K}), \overline{v}_{K}) - \Phi({x}, \overline{y}_{K}; \overline{u}_{K}, \nabla \uh(\overline{y}_{K}))
\label{eq:smooth_csc_primal_dual_cor_eq3}
\end{align}
Therefore summing equations \eqref{eq:smooth_csc_primal_dual_cor_eq1} and \eqref{eq:smooth_csc_primal_dual_cor_eq2}, then setting $u= \nabla f(\overline{x}_K)$, $v=\nabla \uh(\overline{y}_{K})=\nabla h(\overline{y}_K) - \mu_y \overline{y}_K$ and using \eqref{eq:smooth_csc_primal_dual_cor_eq3} we get
\begin{align}
\phi(\overline{x}_{K}, {y}) - \phi({x}, \overline{y}_{K})
&\leq \frac{2L_x}{K(K+1)}\|x-x_{0}\|^2 + \frac{16\|A\|^2}{\mu_y K(K+1)})\|x-x_{0}\|^2 + \frac{2(L_y-\mu_y)}{K(K+1)}\|y-y_{0}\|^2
\label{eq:smooth_csc_primal_dual_cor_eq4}
\end{align}
Finally maximizing both sides over $x \in \cX$ and $y \in \cY$ we get
\begin{align}
\max_{y \in \cY} \phi(\overline{x}_{K}, {y}) - \min_{x \in \cX}  \phi({x}, \overline{y}_{K})
&\leq \frac{2L_x}{K(K+1)}D_\cX^2 + \frac{16\|A\|^2}{\mu_y K(K+1)} D_\cX^2 + \frac{2(L_y-\mu_y)}{K(K+1)}D_\cY^2
\end{align}

(b) Setting $x = x^*$, $y = y^*$, $u=u^{*} = \nabla f(x^*) = \arg\min_u \Ip{x}{u} - f^*(u)$, $v=v^{*} = \nabla \uh(y^*) = \arg\min_v \Ip{y}{v} - \uh^*(v)$ in \eqref{eq:smooth_csc_primal_dual_cor_eq1} we get
\begin{align}
\sum_{k=0}^{K-1} \gamma^k  &[\Phi(x_{k+1}, y^*; u^*, v_{k+1}) - \Phi(x^*, y_{k+1}; u_{k+1}, v^*)] \nonumber \\
\;\leq\; &(L_x + \frac{8\|A\|^2}{\mu_y})\|x^*-x_{0}\|^2 +
(L_y-\mu_y)\|y^*-y_{0}\|^2 - (K+1)^2\frac{\mu_y}{2} \|y^*-y_{K}\|^2
\end{align}
Notice that the LHS above is positive, since by Lemma \ref{lem:minimax_pseudo_gap}, $\Phi(x_{k+1}, y^*; u, v_{k+1}) - \Phi(x^*, y_{k+1}; u_{k+1}, v)
\geq 0$ for all $k=0,1,\ldots$. Then using this fact we get that
\begin{align}
\frac{\mu_y}{4} \|y^*-y_{K}\|^2
\;\leq\;
&(\frac{2L_x}{(K+1)^2} + \frac{16\|A\|^2}{\mu_y(K+1)^2})\|x^*-x_{0}\|^2 +
\frac{2(L_y-\mu_y)}{(K+1)^2}\|y^*-y_{0}\|^2
\label{eq:smooth_csc_primal_dual_cor_eq5}
\end{align}
(c) Let $\hy(x) = \arg\max_{y} \phi(x,y)$, then we can show that $\hy(x)$ is $\|A\|/\mu_y$-Lipschitz continuous in $x$ \cite{nesterov2005smooth}. Then we can show that
\begin{align}
\|\hy(x)-y_{0}\|^2 &\leq 2\|\hy(x)-y^*\|^2 + 2\|y^*-y_{0}\|^2 \nonumber \\
&\leq 2\|\hy(x)-\hy(x^*)\|^2 + 2\|y^*-y_{0}\|^2 \nonumber \\
&\leq 2 \frac{\|A\|^2}{\mu_y^2} \|x-x^*\|^2 + 2\|y^*-y_{0}\|^2
\end{align}
Then using the above inequality and \eqref{eq:smooth_csc_primal_dual_cor_eq4} we get
\begin{align}
\phi_p(\overline{x}_{K}) - \phi_p(x^*) &= \max_{y \in \cY} \phi(\overline{x}_{K}, {y}) - \max_{y \in \cY} \phi(x^*, y) \nonumber \\
&\leq \phi(\overline{x}_{K}, \hy(x)) - \phi(x^*, \overline{y}_{K}) \nonumber \\
&\leq \frac{2L_x}{K(K+1)}\|x^*-x_{0}\|^2 + \frac{16\|A\|^2}{\mu_y K(K+1)}\|x^*-x_{0}\|^2 + \frac{2(L_y-\mu_y)}{K(K+1)}\|\hy(x)-y_{0}\|^2 \nonumber \\
&\leq (\frac{2L_x}{K(K+1)} + \frac{16\|A\|^2}{\mu_y K(K+1)} + \frac{(L_y - \mu_y)}{\mu_y}\frac{4\|A\|^2}{\mu_y K(K+1)} )\|x^*-x_{0}\|^2 + \frac{4(L_y-\mu_y)}{K(K+1)}\|y^*-y_{0}\|^2
\end{align}
From the above inequality it is clear that 
\begin{align}
\phi_p(\overline{x}_{K}) - \phi_p(x^*) &\leq (\frac{4L_x}{K(K+1)} + \frac{32\|A\|^2}{\mu_y K(K+1)} + \frac{(L_y - \mu_y)}{\mu_y}\frac{8\|A\|^2}{\mu_y K(K+1)} )\|x^*-x_{0}\|^2
\end{align}
if
\begin{align}
\|y^*-y_{0}\|^2 \leq (\frac{L_x}{2(L_y-\mu_y)} + \frac{4\|A\|^2}{\mu_y (L_y-\mu_y)} + \frac{(L_y - \mu_y)}{\mu_y}\frac{\|A\|^2}{\mu_y (L_y-\mu_y)} )\|x^*-x_{0}\|^2 \,.
\end{align}
Because of \eqref{eq:smooth_csc_primal_dual_cor_eq5}, we can find a $y_{0}$ satisfying the above inquality by running our algorithm from from $(x_0,y_0)$ for 
\begin{align}
K_0^p \geq \Omega(&\sqrt{{4(L_y-\mu_y)}/{\mu_y}} \times \sqrt{({2L_x} + \frac{16\|A\|^2}{\mu_y})\|x^*-x_{0}\|^2 + {2(L_y-\mu_y)}\|y^*-y_{0}\|^2} \times \nonumber \\ 
&\sqrt{1\bigg/((\frac{L_x}{2} + \frac{4\|A\|^2}{\mu_y } + \frac{(L_y - \mu_y)}{\mu_y}\frac{\|A\|^2}{\mu_y} )\|x^*-x_{0}\|^2 }
)
\end{align}
iterations.
\begin{align}
\|y^*-y_{K_0^p}\|^2 &\leq \frac{4}{\mu_y} (\frac{2L_x}{(K+1)^2} + \frac{16\|A\|^2}{\mu_y(K+1)^2})\|x^*-x_{0}\|^2 +
\frac{4}{\mu_y}  \frac{2(L_y-\mu_y)}{(K+1)^2}\|y^*-y_{0}\|^2 \nonumber \\
&\leq (\frac{L_x}{2(L_y-\mu_y)} + \frac{4\|A\|^2}{\mu_y (L_y-\mu_y)} + \frac{(L_y - \mu_y)}{\mu_y}\frac{\|A\|^2}{\mu_y (L_y-\mu_y)} )\|x^*-x_{0}\|^2
\end{align}

Similarly using the above inequality and \eqref{eq:smooth_csc_primal_dual_cor_eq4} we get
\begin{align}
\phi_d(y^*) - \phi_d(\overline{y}_{K}) &= \min_{x \in \cX} \phi(x, y^*) - \min_{x \in \cX} \phi(x, \overline{y}_{K}) \nonumber \\
&\leq \phi(\overline{x}_{K}, y^*) - \min_{x \in \cX} \phi(x, \overline{y}_{K}) \nonumber \\
&\leq \frac{2L_x}{K(K+1)} D_\cX^2 + \frac{16\|A\|^2}{\mu_y K(K+1)}D_\cX^2 + \frac{2(L_y-\mu_y)}{K(K+1)}\|y^*-y_{0}\|^2
\end{align}

From the above inequality it is clear that 
\begin{align}
\phi_p(\overline{x}_{K}) - \phi_p(x^*) &\leq \frac{4L_x}{K(K+1)} D_\cX^2 + \frac{32\|A\|^2}{\mu_y K(K+1)}D_\cX^2
\end{align}
if
\begin{align}
\|y^*-y_{0}\|^2 \leq (\frac{L_x}{(L_y-\mu_y)} + \frac{8\|A\|^2}{\mu_y (L_y-\mu_y)} )D_\cX^2 \,.
\end{align}
Because of \eqref{eq:smooth_csc_primal_dual_cor_eq5}, we can find a $y_{0}$ satisfying the above inquality by running our algorithm from from $(x_0,y_0)$ for 
\begin{align}
K_0^d \geq \Omega(&\sqrt{{4(L_y-\mu_y)}/{\mu_y}} \times \sqrt{({2L_x} + \frac{16\|A\|^2}{\mu_y})\|x^*-x_{0}\|^2 + {2(L_y-\mu_y)}\|y^*-y_{0}\|^2} \times \nonumber \\ 
&\sqrt{1\bigg/(({L_x} + \frac{8\|A\|^2}{\mu_y } )D_\cX^2 }
)
\end{align}
iterations.
\begin{align}
\|y^*-y_{K_0^d}\|^2 &\leq \frac{4}{\mu_y} (\frac{2L_x}{(K+1)^2} + \frac{16\|A\|^2}{\mu_y(K+1)^2})\|x^*-x_{0}\|^2 +
\frac{4}{\mu_y}  \frac{2(L_y-\mu_y)}{(K+1)^2}\|y^*-y_{0}\|^2 \nonumber \\
&\leq (\frac{L_x}{(L_y-\mu_y)} + \frac{8\|A\|^2}{\mu_y (L_y-\mu_y)} )D_\cX^2
\end{align}
\end{proof}

\section{Balanced Mirror-Prox and Additional experimental details for Section \ref{sec:expts}}
\label{sec:expt_details}
 
For all the experiments we used the theory specified stepsize choices. Balanced Mirror Prox (which we shorten as MP Bal.) is variant of the standard Mirror-Prox algorithm (folklore). For implementing MP Bal.~first we normalize the distance functions so that objective becomes $1$-strongly convex in both the min variable $x$ and the max variable $y$. This modifies Lipschitz constants of the gradients as $L_x \gets L_x/\mu_x $, $L_{xy} \gets L_{xy}/\sqrt{\mu_x \mu_y} = \|A\|/\sqrt{\mu_x \mu_y}$, $L_y \gets L_y/\mu_y$. Finally, in this modified geometry (distance metrics), we run the standard MP with the stepsize $1/\max(L_x, L_{xy}, L_y)$. Since we modified the Lipschitz constants of the gradients this leads to a iteration complexity of $\Ord(\sqrt{\frac{L_x}{\mu_x}} + \frac{\|A\|}{\sqrt{\mu_x\mu_y}} + \sqrt{\frac{L_y}{\mu_y}}) \log(\frac1{\varepsilon})$. This result was also mentioned as a known folklore in Appendix C of \cite{cohen2021relative}.

For experiments using quadratic minimax problems we use $d=5$ and we generate $B$, $A$, $C$ as follows. Let $\Lambda = \mathrm{diag}(r^{0}, r^{1}, \ldots, r^{d-1}) $. Then $A = Q^{(A,2)} \Lambda (Q^{(A,1)})^\top$, $B = \widetilde{B}^\top \widetilde{B}$, $\widetilde{B} = Q^{(B,1)} \Lambda (Q^{(B,2)})^\top$, $C = \widetilde{C}^\top \widetilde{C}$, $\widetilde{C} = Q^{(C,1)} \Lambda (Q^{(C,2)})^\top$, where $Q^{(A,1)}$, $Q^{(A,2)}$, $Q^{(B,1)}$, $Q^{(B,2)}$, $Q^{(C,1)}$, $Q^{(C,2)}$ are i.i.d.~$d \times d$ orthonormal matrices which are generated uniformly at random.
For Figure \ref{fig:synthetic_scsc} we set $r=2.0$, and for Figure \ref{fig:synthetic_kappa} we vary $r$ using the values $\{1.25, 1.5, 1.75, 2.0, 2.25\}$.

\end{document}